\documentclass[11pt,a4paper,twoside]{article}

\title{Phase transitions for interacting particle systems on random graphs}
\date{}
\author{Benedetta Bertoli$^1$, Grigorios A. Pavliotis$^1$, Niccolò Zagli$^{1,2}$}
\usepackage{graphicx}
\usepackage[a4paper]{geometry}
\graphicspath{ {./images/} }
\usepackage{framed}
\usepackage{amsmath}
\usepackage{amssymb}
\usepackage{amsfonts}
\usepackage{mathrsfs}
\usepackage{array}
\usepackage{enumerate}
\usepackage{amsthm}  
\usepackage{tikz}
\usepackage[linktocpage=true]{hyperref}
\usepackage{xcolor}
\usepackage{dsfont}
\usepackage{wrapfig}
\usepackage{float}
\usepackage{setspace}
\usepackage{mathtools}
\usepackage{hyperref}
\usepackage{caption}
\usepackage{subcaption}
\usepackage{verbatim}
\usepackage{fullpage}
\usepackage{hyperref}

\theoremstyle{plain}
\newtheorem{thm}{Theorem}[section] 
\newtheorem{lem}[thm]{Lemma} 
\newtheorem{prop}[thm]{Proposition} 
 
\newtheorem{assumption}[thm]{Assumption}
 
\theoremstyle{definition} 
\newtheorem{defn}[thm]{Definition}

\theoremstyle{remark}
\newtheorem{rem}[thm]{Remark} 
\newcommand{\R}{\mathbb{R}}

\newcommand{\F}{\mathcal{F}}

\newcommand{\dd}{\,\mathrm{d}}

\newcommand{\ML}{\mathcal{M}_1^L([0,2\pi] \times [0,1])}

\begin{document}

\maketitle
{\small \noindent $^1$Department of Mathematics, Imperial College London, London SW7 2AZ, UK,
\\
benedetta.bertoli21@imperial.ac.uk,  pavl@ic.ac.uk and n.zagli18@imperial.ac.uk
\\
$^2$School of Computing and Mathematical Sciences, University of Leicester, LE17RH, Leicester, UK
}

\begin{abstract}
 In this paper,  we study weakly interacting diffusion processes on random graphs. Our main focus is on the properties of the mean-field limit and, in particular, on the nonuniqueness and bifurcation structure of stationary states. By extending classical bifurcation analysis to include multichromatic interaction potentials and random graph structures, we explicitly identify bifurcation points and relate them to the spectral properties of the graphon integral operator. In addition, we develop a self-consistency formulation of stationary states that recovers the primary critical threshold and reveals secondary bifurcations along non-uniform branches. Furthermore, we characterize the resulting McKean-Vlasov PDE as a gradient flow with respect to a suitable metric. In addition, we provide strong evidence that (minus) the interaction energy of the interacting particle system serves as a natural order parameter. In particular, beyond the transition point and for multichromatic interactions, we observe an energy cascade that is strongly linked to dynamical metastability. 
\end{abstract}

\section{Introduction}

\subsection{General introduction}

The study of stochastic interacting particle systems (SIPS) and their mean-field limit has been a topic of extensive research in recent decades due to their wide-ranging applications in physics~\cite{FrankBook, RevModPhys.77.137, Ruffo_al_2018}, biology~\cite{Painter_al_2024, Degond_al_2017}, and even the social sciences~\cite{Naldi-al-2010, Chazelle_al2017a, Pavliotis_al_2021, CAHILL2025130490, helfmann2023modelling}. One of the most interesting aspects of such systems is the emergence of collective behavior at the macroscale, due to the interaction between the particles at the microscale. Examples include the emergence of consensus in models for opinion dynamics~\cite{Pavliotis_al_2021, wehlitz2024approximating}, chemotaxis~\cite{Stevens_2000}, collective synchronization~\cite{Ruffo_al_2018}, emergence of order in active matter~\cite{chepizhko2021revisiting}, self-organised alignment dynamics~\cite{Degond2015}, mean-field games and macroeconomics~\cite{Degond_al_2014},  synchronization dynamics in biological and technological systems~\cite{SmithGottwald_2020, SmithGottwald_2021}. In recent years, it has been recognized that the emergence of collective behavior at the macroscale can be interpreted as a disorder-order phase transition~\cite{carrillo2020long}. The area of SIPS and of their mean-field limit has experienced enormous progress in recent years, and many extensions to, e.g. multi-species models~\cite{Giunta_2022} or moderately interacting diffusions~\cite{hao2024propagationchaosmoderatelyinteracting} have been made.

In many applications, it is important to consider interacting particle systems and, more generally, agent-based models, on graphs. Interacting multi-agent systems on graphs have been used in several different applications, such as biology and the social sciences. Examples include the dynamics of power grid networks~\cite{gaskin2024inferring}, opinion dynamics \cite{during2024breaking, aletti2022opinion, frasca2024opinion, nugent2025opinion}, models of biological neurons \cite{jabin2021mean, jabin2023mean}, social networks~\cite{Porter_2016}, mean-field (stochastic) games \cite{caines2021graphon, neuman2024stochastic}.  

In recent years, the rigorous analysis of interacting particle systems on graphs has seen significant development in cases where white noise is incorporated into the system. Stochastic interacting particle systems (SIPS) on random graphs have been studied in many recent works, including \cite{bayraktar2020graphon, bhamidi2019weakly, delattre2016note, bet2024weakly, luccon2020quenched, coppini2022long, coppini2022note, gkogkas2022graphop, crucianelli2024interacting}. These studies have established, under appropriate assumptions on the interaction potential and on the graph structure, several results on propagation of chaos, the law of large numbers, and the central limit theorem. A crucial element of the analysis presented in these (and many other) papers is the systematic use of the theory of graphons~\cite{Lovasz_2012, medvedev2018continuum}; see also recent developments on network dynamics on graphops~\cite{Kuehn_2020}.  

Clearly, different network topologies can have a profound impact on the collective dynamics of the SIPS. This is a topic that has been studied extensively in the past few years, in particular, for Kuramoto-type models and their variants. Kuramoto-type models on graphs have been extensively studied in the deterministic setting in a series of papers by Medvedev and collaborators (\cite{chiba2016mean}, \cite{chiba2017mean},\cite{medvedev2018continuum}, \cite{kaliuzhnyi2018mean}, \cite{chiba2018bifurcations}). In these important papers, the authors obtained several results on the long-term behaviour and bifurcations of such systems; in particular, the influence of the network topology on the synchronization onset was studied. Both first-order (overdamped) and second-order (inertial) dynamics were considered.

It is well known that, for fully coupled SIPS at equilibrium in the absence of an underlying network topology, the mean-field dynamics can exhibit several stationary states, corresponding to critical points of the system's free energy. These dynamics are characterized by phase transitions, where changes in control parameters, such as temperature and/or interaction strength, lead to shifts in the global minimizer of the free energy. A standard example is the disorder-order phase transition from the uniform to the synchronized state in the noisy Kuramoto model~\cite{Giacomin_al_2010}. 

However, for interaction potentials with several nonzero Fourier modes (so-called multichromatic potentials) the nature of the stationary states is substantially richer. There are several applications, e.g. in polymer dynamics~\cite{LuVuk_2010} and biology~\cite{Primi2009} where it is necessary to consider interaction potentials (on the torus or the sphere) with several nonzero (negative) Fourier modes. As examples, we mention the Onsager and Maiers-Saupe potentials~\cite{carrillo2020long, LuVuk_2010, vukadinovic2023phase} or the Hegselmann-Krause interaction potential~\cite{Chazelle_al2017a, Pavliotis_al_2021} used in opinion dynamics models. A recent comprehensive study of stationary solutions in McKean–Vlasov systems with more general interaction potentials, but without a graphon structure, can be found in~\cite{balasubramanian2025structure}.

The main goal of this paper is to systematically study phase transitions for SIPS on random graphs, with particular emphasis on how the interaction potential and the network topology influence the nature of these transitions. 
In particular, we extend the bifurcation analysis presented in the aforementioned works beyond the Kuramoto model, i.e. we consider the multichromatic interaction potentials studied in~\cite{bertoli2025} and consider a variety of network topologies. In~\cite{bertoli2025} the stability of multipeak solutions was studied in detail. In particular, it was shown that, in general, such states tend to be unstable and that the dynamics converges, in the long-time limit, to a single cluster. Understanding the stability of cluster/multipeak states is a very interesting problem related to the phenomenon of dynamical metastability~\cite{Butta_al_2003,BashiriMenz2021}.  
Here, we build on these results by developing a self-consistency formulation for stationary states of the McKean-Vlasov equation on graphs. This approach recovers the primary synchronization threshold and further identifies the onset of secondary instabilities of multipeak states.
One of the questions that we address in this paper is whether multipeak solutions are unstable for a variety of random graph topologies. In addition, we introduce an "order parameter", namely the negative of the interaction energy, that keeps track of the number of clusters of the SIPS. The dynamics of clusters for interacting particle systems with local attractive interactions--leading to first order phase transitions in the mean field limit--in the absence of a graphon structure, was studied recently in~\cite{gerber2025formationclusterscoarseningweakly,leimkuhler2025clusterformationdiffusivesystems}.


Our main contributions are as follows.
\begin{itemize}
    \item We identify the gradient flow structure of the McKean-Vlasov equation with a modified Wasserstein metric. The associated free energy functional characterizes stationary states and provides insight into long-time dynamics.
    \item Using the Crandall-Rabinowitz theorem, we derive explicit formulas for critical interaction thresholds for synchronization transitions on various random graphs. Our approach, inspired by Tamura~\cite{tamura1984asymptotic}, complements linear stability analysis and highlights the role of network topology.
    \item We introduce a self-consistency formulation of stationary states that provides a closed system for the Fourier modes of the equilibrium density. In particular, we show that beyond the primary transition from the uniform state, new branches of non-uniform equilibria may themselves lose stability. We study how the underlying network topology affects the nature of these secondary bifurcations. 
    \item We perform comprehensive numerical simulations of $N$-particle systems on various graph topologies, including the Erdős-Rényi, small-world and power-law graphs. Using the interaction energy as an order parameter, we identify phase transitions and analyse the impact of dynamical metastability on the dynamics of SIPS with multichromatic interaction potentials. 
    
\end{itemize}

\subsection{Set-up}\label{subsec: setup}
In this section, we introduce the system of weakly interacting diffusions together with the class of random graphs on which the interactions take place.
The random graph structure is described in terms of a symmetric measurable function $W : [0,1]^2 \rightarrow [0,1]$, referred to as a graphon~\cite{kaliuzhnyi2018mean}. To construct a discrete approximation of $W$ for a system of $N$ particles we adopt the following standard procedure.
First, we partition $[0,1]$ into equal subintervals of length $\frac{1}{N}$, $I_{N,i} := [\frac{i-1}{N}, \frac{i}{N})$, for $i = 1, \dots, N$.
Then, for each pair $i,j \in [N] := \{ 1, \ldots, N \}$, we define the edge weight:
\begin{align*}
    W_{N,ij} := N^2 \int_{I_{N,i} \times I_{N,j}} W(x,y) \dd x \dd y
\end{align*}
to be the average value of $W(x,y)$ over the interval $I_{N,i} \times I_{N,j}$. Using these values, we define the piecewise constant approximation:
\begin{align*}
    W_N(x,y) = W_{N,ij} \mathds{1}_{(x,y) \in I_{N,i} \times I_{N,j}}.
\end{align*}
This converges to $W(x,y)$ almost everywhere and in $L^2([0,1]^2)$ as $N \to \infty$; see \cite[Lemma 3.3]{kaliuzhnyi2018mean}.
This gives rise to a weighted random graph $\Gamma_N = (V_N, E_N)$ by setting the vertex set $V_N := [N]$ and the edge set as:
\begin{align*}
    E(\Gamma_N) :=\{ \{i,j\} : W_{N,ij} \neq 0, i,j \in [N] \}.
\end{align*}

For $N \in \mathbb{N}$, we consider the system of stochastic differential equations (SDEs) on $\Gamma_N$ given by:
\begin{equation}\label{SDE}
    \dd X_t^i = -\frac{\theta}{N \alpha_N} \sum_{j=1}^N W_{N,ij} D'(X_t^i - X_t^j) \dd t + \sqrt{2\beta^{-1}} \dd B_t^i,
\end{equation}
where $X_t^i \in [0,2\pi], i = 1, \ldots, N$ represent the positions of the $N$ particles and $B_t^i, i = 1, \ldots, N$ are standard independent Brownian motions.
Here, $\theta>0$ represents the strength of the interaction and $\beta$ is the inverse temperature. Furthermore, the $N$-dependent scaling factor $\alpha_N$ is introduced to guarantee a non-trivial limit of equations \eqref{SDE} for $N \to \infty$ when the underlying network is sparse; see ~\cite{chiba2018bifurcations} and later discussion. For dense graphs, it can be assumed that $\alpha_N = 1$. $D:[0,2\pi] \rightarrow \R$ is the interaction potential.\footnote{In this paper, we will consider the SIPS on graphs in one dimension.} We will assume that $D$ is an even $2\pi$-periodic Lipschitz continuous function.

As $N \to \infty$, the empirical measure associated with the system \eqref{SDE} converges, under appropriate initial conditions, to a probability density $\rho = \rho(t,u,x)$ satisfying the nonlinear Fokker-Planck (McKean-Vlasov) equation:
\begin{equation}\label{FP}
    \partial_t \rho(t,u,x) = \theta \partial_u \left( \rho(t,u,x) \int_0^1 \int_0^{2\pi} W(x,y) D'(u-v) \rho(t,v,y) \dd v \dd y \right) + \beta^{-1} \partial_u^2 \rho(t,u,x),
\end{equation}

The rigorous derivation of~\eqref{FP} and propagation of chaos for systems like~\eqref{SDE} has been established in a number of recent works, under various assumptions on $W$ and $D$~\cite{coppini2022note, gkogkas2022graphop, delattre2016note, bayraktar2022stationarity, bayraktar2023graphon, luccon2020quenched}. In particular:
- \cite[Prop. 1.3]{coppini2022note} proves well-posedness assuming measurability of the law and control on second moments,
- \cite[Prop. 2.4]{luccon2020quenched} considers differentiable $D$ with sublinear growth and regular $W$,
- \cite[Thm. 3.2]{bet2024weakly} treats bounded Lipschitz potentials $D$ and establishes uniqueness.


The purpose of this paper is to study the long-term behavior of both the $N$- particle system  \eqref{SDE} and the mean-field limit \eqref{FP}, for a variety of interaction potentials $D$ and graphons $W$. In particular, we will focus on the following random graphs:
\begin{itemize}
    \item Erdős-Rényi (ER) graph. The ER graph is a dense graph constructed by setting an edge between all nodes $i$ and $j$ with constant probability $p \in [0,1]$. The adjacency matrix of the ER graph is $W_{N,ij} = 1$ if there is an edge between nodes $i$ and $j$, and $W_{N,ij}  =0$ otherwise. It follows that the graphon associated with an ER graph is the constant function $W(x,y) = p$. When $p = 1$, the ER graph corresponds to the all-to-all, complete graph. 
    \item Small-World (SW) graph. First introduced in \cite{Watts_Strogatz1998}, the SW graph is a dense graph that interpolates between an $r$-nearest neighbours ring lattice and an ER graph. The resulting graph structure is quite regular but allows for random edges across the network. Its name derives from the property of such graphs of having short path lengths between nodes originally far located on the ring. There are numerous variants of the original algorithm described in \cite{Watts_Strogatz1998} to construct SW graphs. We here use the following method.  Consider $N$ nodes arranged on a ring, with each node having $r/2$ neighbours on both sides ($r$ being an even number). Now, select a rewiring probability $p \in [0,1]$. For each node $k$ and all its $r/2$ neighbours on the right, perform, with probability $p$, a rewiring move consisting of creating a new edge between $k$ and a randomly extracted node, provided the edge did not exist before, and destroying the old one. The graphon associated with such a Small World graph is \cite{Medvedev2014,MEDVEDEV2014_SW}
    \begin{equation}\label{e: SW}
        W(x,y) = 
        \begin{cases}
            1 -p +2ph & \text{ if } \min \{ |x-y|, 1 - |x-y| \} \leq h \\
            2ph & \text{ else.}
        \end{cases}
    \end{equation}
    Here, the continuous coupling range $h \in [0,1/2]$ can be estimated as $h = r / (2N)$.  
    
    \item Power-Law (PL) graph. In this paper, we also consider sparse random graphs that have power-law degree distributions. In particular, we consider the power-law graph corresponding to the graphon $W(x,y) = (xy)^{-\gamma}$ for $0 < \gamma < \frac{1}{2}$. For \eqref{FP} to be the mean-field limit of the $N$-particle system \eqref{SDE}, it is fundamental to rescale the interaction strength by a sequence $0 <\alpha_N \leq 1$, satisfying $\alpha_N \to 0$ and $N \alpha_N \to +\infty $. Following~\cite{chiba2018bifurcations}, we here consider $\alpha_N = N^{- \alpha}$ with $0 < \gamma < \alpha < 1$. The graphs corresponding with the power-law graphon are constructed using the procedure described in~\cite{kaliuzhnyi2017semilinear,chiba2018bifurcations}.
    \end{itemize}

For the remainder of the paper, we work under the following assumptions.
\begin{assumption}
    \begin{itemize}
         \item The graphon $W : [0,1]^2 \to [0,1]$ is a symmetric measurable function, i.e., $W(x,y) = W(y,x)$ for all $x, y \in [0,1]$.
        \item $W$ satisfies the following $L^1$ continuity condition: 
        \begin{equation}\label{e: Wregularity}
            \int_0^1 |W(x_1,y)-W(x_2,y)| \dd y \rightarrow 0 \text{ as } |x_1 - x_2| \rightarrow 0.
        \end{equation}
        \item $D : [0,2\pi] \rightarrow \R$ is a Lipschitz continuous, even function.
    \end{itemize}
\end{assumption}
All the graphons mentioned above satisfy \eqref{e: Wregularity}.  For the Small-World graphon, this follows from estimating the measure of the symmetric difference between neighbourhoods on the circle, which varies linearly with $|x_1 - x_2|$; the condition can also be verified for the Erdős–Rényi and Power-Law cases based on their explicit expressions.

\subsection{Organization of the Paper}

The remainder of the paper is organised as follows. In Section \ref{sec:GradientFlow}, we study the gradient flow structure of the Fokker-Planck equation and the associated free energy functional, its key properties and its connection to the existence of stationary states. In Section \ref{sec:BifurcationTheory}, we combine bifurcation theory, spectral analysis, and the self-consistency formulation to derive explicit expressions for the critical interaction strength governing both the primary and secondary transitions on various graph topologies. In Section \ref{sec:NumericalExperiments}, we present numerical simulations of the system of SDEs which validate these theoretical predictions, and illustrate the phase diagrams for multichromatic interaction potentials. In Section \ref{sec:Conclusions}, we summarize our findings.

\section{Free energy and gradient flow structure}\label{sec:GradientFlow}
Before analysing bifurcations, it is useful to understand the variational structure underlying the graphon McKean-Vlasov equation. In particular, the existence of a gradient-flow formulation provides a natural energy functional whose critical points correspond to stationary states.
In the case of a fully connected graph (i.e. $W \equiv 1$), it is known that the McKean-Vlasov equation:
\begin{equation}\label{e: originalFP}
    \partial_t \rho(t,u) = \theta \partial_u \left( \rho(t,u) \int_0^{2\pi} D'(u-v) \rho(t,v) \dd v \right) + \beta^{-1} \partial_u^2 \rho(t,u)
\end{equation}
has a gradient flow structure with respect to the 2-Wasserstein distance (see, for example, \cite{carrillo2020long}).
In particular, Eqn.~\eqref{e: originalFP} can be written as: 
\begin{align*}
    \partial_t \rho = \partial_u \cdot \left(\rho \partial_u \frac{\delta \mathcal{E}}{\delta \rho} \right),
\end{align*}
where $\mathcal{E}$ denotes the free energy
\begin{align}\label{e:free_en}
    \mathcal{E}(\rho, {\beta, \theta}) = \beta^{-1} \int_0^{2\pi} \rho(u) \log (\rho(u)) \dd u + \frac{{\theta}}{2} \int_0^{2\pi} \int_0^{2\pi} D(u-v) \rho(u) \rho(v) \dd u \dd v. 
\end{align}
Our goal in this section is to extend this gradient flow structure to the graphon setting. The key difference is that the graphon McKean-Vlasov PDE \eqref{FP} explicitly depends on the spatial variable $x \in [0,1]$, but involves no derivatives with respect to it. As such, the standard $W_2$ metric is no longer appropriate. To resolve this, we adapt the modified Wasserstein structure introduced in \cite{bashiri2020gradient}, treating the $x$-variable as a fixed label and defining transport only in the $u$-direction for each $x$. This leads to a modified space of probability densities, defined as follows.
Let $\ML$ denote the set of probability measures on $[0,2\pi] \times [0,1]$ whose marginal in $x$ is the Lebesgue measure. We define
\begin{align*}
  \mathcal{P}^L_2([0,2\pi] \times [0,1]) := \left\{ \mu \in \mathcal{M}_1^L([0,2\pi] \times [0,1]) : \int_{[0,2\pi] \times [0,1]} |u|^2 \dd \mu(u, x) < \infty \right\}.
\end{align*}
equipped with the localised Wasserstein metric:
\begin{align*}
    (W^L(\mu, \nu))^2 := \int_0^1 W_2(\mu^x, \nu^x)^2 \dd x,
\end{align*}
where $\mu = \mu^x \dd x$.

The associated free energy functional in this setting becomes:
\begin{align}
\label{eq: Free energy}
    \mathcal{F}(\rho) = \beta^{-1} \iint \rho(u,x) \log \rho(u,x) \dd u \dd x + \frac{\theta}{2} \iiiint W(x,y) D(u - v) \rho(u,x) \rho(v,y) \dd u \dd x \dd v \dd y.
\end{align}
The transport dynamics under $W^L$ are local in $x$, and, for each fixed $x$, the marginal $\rho^x(u)$ evolves according to a local continuity equation:
\begin{equation}\label{e:local_continuity}
    \partial_t \rho^x(u) + \partial_u (\rho^x(u) v^x(u)) = 0,
\end{equation}
where $v^x(u)$ is the velocity field minimizing the free energy in the $u$-Wasserstein geometry for fixed $x$.
In particular, 
\begin{align*}
    v^x(u) = -\partial_u \frac{\delta \mathcal{F}}{\delta \rho}(u,x).
\end{align*}
For a formal derivation of the continuity equation in a similar setting, we refer to \cite[Sec. 3]{bashiri2020gradient}.
We now compute the first variation of $\F$ with respect to $\rho$, verifying that this gradient recovers the PDE \eqref{FP}.
We recall that $\frac{\delta \mathcal{F}}{\delta \rho}$ is any measurable function that satisfies:
\begin{align*}
    \frac{\dd}{\dd \varepsilon} \mathcal{F}(\rho+\varepsilon \rho_1) \bigg\rvert_{\varepsilon =0} = \int \frac{\delta \mathcal{F}}{\delta \rho}(\rho) \dd \rho_1
\end{align*}
for all smooth perturbations $\rho_1$. We consider each term of $\F = \mathcal{S} + \mathcal{E}_{\text{int}}$.

\paragraph{Entropy term.} The entropy term is given by
\begin{align*}
    \mathcal{S}(\rho) = \beta^{-1} \iint \rho(u,x) \log \rho(u,x) \dd u \dd x.
\end{align*}
By differentiating under the integral, we obtain the standard result:
\begin{align*}
    \frac{\delta \mathcal{S}}{\delta \rho}(u,x) = \beta^{-1} (1+\log \rho(u,x)).
\end{align*}
\paragraph{Interaction energy.} The interaction energy is:
\begin{align}
\label{eq: interaction energy}
    \tilde{E} = \frac{\theta}{2} \iiiint W(x,y) D(u - v) \rho(u,x) \rho(v,y) \dd u \dd x \dd v \dd y.
\end{align}
To compute the first variation of this, we expand $\tilde{E}(\rho + \varepsilon \rho_1)$ to first order in $\varepsilon$:
\begin{align*}
    \tilde{E} (\rho + \varepsilon \rho_1) = \tilde{E}(\rho) + \varepsilon \iiiint W(x,y) D(u-v) \rho_1(u,x) \rho(v,y) \dd u \dd x \dd v \dd y + O(\varepsilon^2).
\end{align*}
This gives the variational derivative:
\begin{align*}
    \frac{\delta \tilde{E}}{\delta \rho} (u,x) = \theta \iint W(x,y)D(u-v) \rho(v,y) \dd v \dd y
\end{align*}

Putting both terms together:
\begin{align*}
    \frac{\delta \mathcal{F}}{\delta \rho}(u, x) =  \beta^{-1}(1 + \log(\rho)) + \theta \int_0^1 \int_0^{2\pi} W(x,y) D(u-v) \rho(v,y) \dd v \dd y .
\end{align*}

Substituting this into the continuity equation \eqref{e:local_continuity}, we recover the graphon Fokker-Planck equation.

\subsection{Properties of the Free Energy}
A critical advantage of the gradient-flow formulation of \eqref{FP} is that it provides us with a free energy functional $\mathcal{F}(\rho)$ whose critical points correspond to stationary states of the system. Before moving onto the bifurcation analysis in section \ref{sec:BifurcationTheory}, we briefly examine the structure of $\mathcal{F}$ and its convexity properties, as they offer a useful qualitative picture.
\begin{prop}\label{p: critical point}
    $\rho$ is a stationary solution of \eqref{FP} if and only if it is a critical point of the free energy.
\end{prop}
\begin{proof}
    The proof is an adaptation of ~\cite[Prop. 2.4]{carrillo2020long}, and it relies on the formulation of stationary states given by Theorem \ref{t: stationary distribution}. In particular, $\rho$ solves:
    \begin{equation}\label{e:stationaryrho}
        \rho = \frac{1}{Z(\rho)} \exp \left( -\beta \theta \int_0^1 \int_0^{2\pi} W(x,y) D(u-v)\rho(v,y) \dd v \dd y \right).
    \end{equation}
    Using this formulation, one can verify that for $\rho, \rho_1 \in \mathcal{P}_2^L, s \in [0,1]$, $\frac{\dd}{\dd s} \mathcal{F}(\rho_s) \bigg\rvert_{s=0} = 0$ for any $\rho_1$, where $\rho_s = (1-s)\rho + s \rho_1$.
\end{proof}
The sign of the interaction potential $D$ plays a key role in shaping the energy landscape. We introduce the notion of $H$-stability, which provides a simple criterion for whether $\F$ is convex.
\begin{defn}($H$-stable potential)
    We say the interaction potential $D:[0,2\pi] \rightarrow \R$ is $H$-stable if for every bounded signed measure $\mu$, we have $\int_0^{2\pi} \int_0^{2\pi} D(u-v) \mu(\dd u) \mu (\dd v) \geq 0$. Equivalently, $D$ is $H$-stable if its Fourier transform $\hat{D}(k) = \int_0^{2\pi} D(u) e^{iku} \dd u$ is non-negative almost everywhere for every $k \in \mathbb{Z}$.
\end{defn}
Every potential $D$ can be decomposed into a stable part $D_s$ consisting of the positive Fourier modes of $D$, and a remaining unstable part $D_u$ (\cite[Sec. 2.2]{carrillo2020long}).
\begin{prop}\label{p: F convex}
    Assume $D$ is $H$-stable. Then, the free energy functional $\mathcal{F}(\rho)$ is convex.
\end{prop}
\begin{proof}
    The entropy term $S(\rho)$ of the free energy is convex because the function $\rho \mapsto \rho \log\rho$ is convex for $\rho >0$. 
    To show convexity for $\mathcal{E}_{\text{int}}(\rho)$, we observe that it is a quadratic form with kernel $K(u,x,v,y) = \theta W(x,y)D(u-v)$. If $D$ is of positive type, and $W \geq 0$, then $K$ is positive semi-definite. This implies that $\mathcal{E}_{\text{int}}$ is convex. Therefore, $\mathcal{F} = \mathcal{S}+\mathcal{E}_{\text{int}}$ is convex.
\end{proof}

\section{Bifurcation theory for the McKean-Vlasov PDE on graphs}\label{sec:BifurcationTheory}
The McKean-Vlasov equation on graphs can exhibit qualitative changes in behaviour as the interaction strength $\theta$ varies. In this section, we analyse the emergence of non-uniform stationary states as $\theta$ crosses a critical threshold. Our approach combines bifurcation theory, spectral analysis, and a self-consistency formulation of stationary states. The discussion is organised as follows:
\begin{enumerate}
    \item We describe the structure of stationary solutions and establish their existence and uniqueness for small $\theta$, following the method in \cite{tamura1984asymptotic}.
    \item We characterise primary bifurcations from the uniform state via spectral analysis and compute the corresponding critical value $\theta_c $.
    \item We reformulate the stationary problem as a system of self-consistency equations for the Fourier modes, which both recovers the primary threshold and allows the study of secondary bifurcations on non-uniform branches, giving rise to a second critical value $\theta_c^{(2)}$.
\end{enumerate}
We summarise the main result here: 
\begin{thm}
    Assume $D(u) = -\sum_{k \geq 1} a_k \cos(ku)$ with $a_k >0$. Let $L$ be the graphon integral operator $(Lv)(x) = \int_0^1 W(x,y) v(y) \dd y$, with eigenvalues ${\lambda}_{l \in \mathbb{N}}$ and largest eigenvalue $\lambda_1 >0$. As the interaction strength $\theta$ increases, the system undergoes a bifurcation from the uniform to a non-uniform stationary state.
    \begin{enumerate}
        \item \textbf{Primary bifurcation:} For each Fourier mode $m \geq 1$ and graphon eigenvalue $\lambda_l$,
        \begin{align*}
            \theta_{m,l} = \frac{2}{\beta \lambda_l a_m}.
        \end{align*}
        The first (primary) threshold is
        \begin{align*}
            \theta_c  = \min_{m,l} \theta_{m,l} .
        \end{align*}
        In particular, if $\lambda_1$ is simple and $a_{m_*} = \max_k a_k$, then:
        \begin{align*}
            \theta_c  = \frac{2}{\beta a_{m_*} \lambda_1}.
        \end{align*}
        \item \textbf{Secondary bifurcation:}
        Consider the bichromatic potential $D(u) = -a_1\cos(u)-a_2\cos(2u)$, $a_2 > a_1$. Let $R_{2}(\cdot;\theta)$ denote the order parameter of the even branch. Then, near the bifurcation, $R_{2}(x)$ is proportional to the leading eigenfunction $\phi_1$ of $L$. Define:
        \begin{align*}
            h_{2}(x;\theta) = \beta \theta a_{2} R_{2}(x;\theta), \qquad g(x;\theta) = \frac{1}{2} \left(1+\frac{I_1(h_2(x;\theta))}{I_0(h_2(x;\theta))} \right)
        \end{align*}
        Set $A(\theta) := L \circ M_{g}$, where $M_g$ is multiplication by $g$. Then, the odd mode $m=1$ becomes unstable at the unique $\theta = \theta_c^{(2)}$ solving:
        \begin{align*}
            \beta \theta_c^{(2)} a_1 \lambda_{\max}(A (\theta_c^{(2)})) = 1.
        \end{align*}
        Moreover,
        \begin{align*}
            \frac{1}{\beta a_1 \lambda_1} \leq \theta_c^{(2)} \leq \frac{2}{\beta a_1 \lambda_1}.
        \end{align*}
        In the special case where the leading eigenfunction of $L$ is piecewise constant (e.g. Erdős–Rényi, Small-World), one has $A(\theta) = g(\theta) L$, and hence:
        \begin{align*}
            \theta_c^{(2)} = \frac{1}{\beta a_1 g(\theta_c^{(2)})\lambda_1}.
        \end{align*}
    \end{enumerate}
\end{thm}
\begin{rem}
    We expect that the second part of the theorem holds for the more general class of interaction potentials considered in the first part. However, the formulas for $g$ and $\theta_c^{(2)}$ become more complicated and we refrain from showing the details here.   
\end{rem}
\subsection{Stationary solutions as fixed points}
To analyse bifurcations of stationary states, it is convenient to rewrite the McKean-Vlasov equation \eqref{FP} as a non-linear fixed-point problem for the stationary density.
Let $\rho(u,x) \in L^1([0,2\pi] \times [0,1])$ denote a probability density. A stationary solution of the McKean-Vlasov equation \eqref{FP} satisfies $\rho = f(\rho, \theta)$, where the map $f:L^1([0,2\pi] \times [0,1]) \rightarrow L^1([0,2\pi] \times [0,1]) $ is defined by:
\begin{equation}\label{e:f_map}
  f(\rho,\theta) := \frac{1}{Z(\rho)} \exp \left( -\beta \theta \int_0^1 \int_0^{2\pi} W(x,y) D(u-v)\rho(v,y) \dd v \dd y \right),
\end{equation}
and $Z$ is the normalisation constant.
In particular, the constant solution $\rho \equiv \frac{1}{2\pi}$ always solves the equation.




\begin{thm}\label{t: stationary distribution}(Characterisation of stationary states)
    \begin{enumerate}
        \item For any $\theta \in \R$, there exists a stationary solution to \eqref{FP}.
        \item For $|\theta|$ sufficiently small, \eqref{FP} admits a unique stationary distribution.
    \end{enumerate}
\end{thm}

\begin{proof}
    We sketch the argument for each part.
    (1) The existence of stationary states follows from Schauder's fixed point theorem. Define the closed, convex subset $B:=\{ \rho \in L^1([0,2\pi] \times [0,1]): \rho \geq 0, \|\rho \|_1 =1 \}$. The map $f(\rho,\theta)$ is continuous from $B \rightarrow B$. Due to the regularity assumptions on $W$ and $D$, its image is relatively compact. Hence, we can apply the Arzelà–Ascoli theorem and dominated convergence to an appropriate set of functions to prove compactness. Therefore, $f$ admits a fixed point in $B$.

    (2) For small $|\theta|$, the map $f$ becomes a contraction. Using a standard inequality for exponentials,
    \begin{align*}
        |e^a - e^b| \leq e^{\max(a,b)} |a-b|,
    \end{align*}
    we can estimate the difference $\|f(\rho_1) - f(\rho_2)\|_1$ in terms of $\| \rho_2 - \rho_1\|_1$. This gives a Lipschitz bound $C_{\theta} < 1$ when $\theta$ is small, establishing uniqueness by Banach's fixed-point theorem.
\end{proof}
\subsection{Bifurcation from the uniform state}\label{subsec: bifurcation}
We rewrite the stationary equation as $g(\rho,\theta) := \rho - f(\rho,\theta) =0$, and analyse bifurcations from the uniform state $\rho = \frac{1}{2\pi}$.
We define the operator:
\begin{align*}
    T\rho(u,x) =- \frac{\beta}{2\pi} \int_0^1 \int_0^{2\pi} W(x,y) D(u-v) \rho(v,y) \dd v \dd y.
\end{align*}
We restrict to studying this operator on the subspace of even probability densities, i.e. functions $\rho(u,x)$ satisfying $\rho(-u,x) = \rho(u,x)$. This is crucial as, when restricted to even functions, the operator $T$ admits simple eigenvalues under mild assumptions on $W$ and $D$. Without this restriction, symmetry-induced multiplicities would prevent us from applying standard bifurcation theory, such as the Crandall-Rabinowitz theorem.
\begin{thm}\label{bifurcationthm}(Bifurcation criterion)
    Suppose $\theta_0^{-1}$ is a simple eigenvalue of $T$. Then, $(\rho, \theta) = \left( \frac{1}{2\pi}, \theta_0 \right)$ is a bifurcation point of $g=0$.
\end{thm}

The proof is based on the following characterisation of bifurcation points provided by the Crandall-Rabinowitz theorem(\cite{crandall1971bifurcation}, \cite[Lem. 4.2]{tamura1984asymptotic}).
\begin{lem}[Crandall–Rabinowitz theorem]\label{CR}
    Let $g(\rho,\theta)$ be a smooth map between Banach spaces, and suppose $g \left( \frac{1}{2\pi}, \theta_0 \right) = 0$. Assume that:
    \begin{enumerate}
        \item $D_{\theta} g \left( \frac{1}{2\pi}, \theta_0 \right) =0$.
        \item The linearised operator $D_{\rho} g \left( \frac{1}{2\pi}, \theta_0 \right)$ has a one-dimensional kernel, and its range is closed and of codimension one.
        \item There exists $\phi \in$ ker$D_{\rho} g$ with $D^2_{\rho \theta} g \left(\frac{1}{2\pi}, \theta_0 \right)[\phi] \notin \text{Im} D_{\rho} g$.
    \end{enumerate}
    Then, $\left( \frac{1}{2\pi}, \theta_0 \right)$ is a bifurcation point of the equation $g(\rho,\theta) =0$.
\end{lem}

\subsection{Spectral characterization and critical threshold}\label{subsec: threshold}
To compute the critical interaction strength $\theta_c$, we study the spectral properties of the operator $T \rho(u,x)$. Our goal is to determine when the operator $T$ has a simple eigenvalue $\theta^{-1}$, as this signals a bifurcation from the uniform state.
We proceed  by diagonalising the graphon integral operator:
\begin{equation}\label{Loperator}
    L[V](x) := \int_0^1 W(x,y) V(y) \dd y.
\end{equation}
$L$ is compact and self-adjoint, and so admits a complete orthonormal basis of eigenfunctions $\{V_l(x)\}_{l \in \mathbb{N}}$ with corresponding eigenvalues $\{ \lambda_l \} \subset \R$.
Since the interaction kernel $D(u)$ is assumed to be even, its Fourier modes are purely cosine terms. Therefore, we consider separable functions of the form:
\begin{align*}
    \rho_{m,l}(u,x) = V_l(x) \cos(mu).
\end{align*}
These functions are even in $u$, and they serve as eigenfunctions for the operator $T$. Plugging $\rho_{m,l}$ into the definition of $T$, one can verify that the corresponding eigenvalues are:
\begin{align*}
    \lambda_{m,l} = -\frac{\beta \lambda_l}{2\pi} \int_0^{2\pi} D(v) \cos(mv) \dd v
\end{align*}
A bifurcation occurs when $\lambda_{m,l} = \theta^{-1}$, provided this eigenvalue is simple. In practice, simplicity depends both on the graphon $W$ and the interaction kernel $D$; for instance, if the integral graphon operator $L$ has degenerate eigenvalues, multiplicity may arise. To account for this, we define the index set:
\begin{align*}
    M := \{ (m,l) \in \mathbb{N}^2 : \lambda_{m,l} \text{ is a simple eigenvalue of } T \}.
\end{align*}
The bifurcation analysis developed in Section \ref{subsec: bifurcation} applies to any $(m,l) \in M$, and the first bifurcation occurs at the pair minimizing the critical value of $\theta$.

We conclude:
\begin{prop}\label{p: multichromatic_critical}
    Under the assumptions above, the system undergoes a bifurcation from the uniform state at the critical interaction strength:
    \begin{equation}\label{e: multichromatic critical}
        \theta_c = \min_{(m,l) \in M} \left[\frac{2 \pi}{\beta {\lambda}_l} \left( \int_0^{2\pi} D(v)\cos(mv) \dd v \right)^{-1} \right].
    \end{equation}
    where $M \subset \mathbb{N}^2$ is the set of index pairs for which $\lambda_{m,l}$ is a simple eigenvalue of the operator $T$.
\end{prop}
For $D(u) = - \sum_{k=1}^n a_k\cos(ku)$, this reduces to $\theta_c = \min_{m,l} \frac{2}{\beta a_m \lambda_l}$.

 \subsection{Stability analysis and free energy expansion}
We can confirm the critical threshold $\theta_c$ for the onset of non-uniform states with two other independent approaches: linear stability and variational analysis. These alternative approaches can offer a more direct route to the critical threshold in other settings where a full bifurcation analysis is not possible or where spectral information is easier to deduce.
\paragraph{Linear stability}
As shown in \cite{gkogkas2022graphop}, linearising the dynamics around the uniform state $\rho \equiv \frac{1}{2\pi}$ leads to a Fourier-mode decomposition in which each mode evolves independently. Stability depends on the eigenvalues of an operator of the form:
\begin{align*}
    F^m w = \frac{1}{2} (-im \theta \hat{D}_{-m} L(w) - m^2 \beta^{-1}w),
\end{align*}
where $\hat{D}_m$ is the $m$-th Fourier coefficient of $D$.
For interaction potentials of the form $D(u) = -\sum_{k=1}^n a_k \cos(ku)$, instability occurs in mode $j \leq n$ if:
\begin{align*}
    \theta > \frac{2}{\beta \lambda_l a_m}
\end{align*}
for an eigenvalue $\lambda_l$ of the graphon operator $L$. This matches the critical value identified in Section \ref{subsec: threshold}.

\paragraph{Free energy second variation}
We can confirm this threshold from a variational perspective by computing the second variation of the free energy of $\F$ around $\rho = \frac{1}{2\pi}$.
To examine stability, we consider small perturbations $\delta \rho(u,x)$ of $\frac{1}{2\pi}$ and compute the second variation:
\begin{align*}
    \delta^2 \F = 2\pi{\beta^{-1}} \int (\delta \rho(u,x))^2 \dd u \dd x+ \theta \int W(x,y) D(u-v) \delta \rho(u,x) \delta \rho(v,y) \dd u \dd x \dd v \dd y.
\end{align*}
To enforce the constraint of mass conservation, we write $\delta \rho = \partial_u q$, where $q(u,x)$ is a mean-zero perturbation potential. Integrating by parts, we obtain:
\begin{align*}
    \delta^2 \F =  \int q(u,x) K(u,v,x,y) q(v,y) \dd u \dd v \dd x \dd y,
\end{align*}
where the kernel is given by:
\begin{align*}
    K(u,v,x,y) = -\theta W(x,y) D''(u-v) - 2\pi \beta^{-1} \partial_u^2 \delta (u-v) \delta (x-y) 
\end{align*}
We therefore consider the eigenvalue problem:
\begin{align*}
    \int_0^1 \int_0^{2\pi} K(u,v,x,y) q(v,y) \dd v \dd y = \lambda q(u,x).
\end{align*}
We work with $D(u) = -\sum_{m=1}^n a_m \cos(mu)$ and let $W$ be a graphon whose integral operator has real eigenvalues $\lambda_l$ with eigenfunctions $V_l$:
\begin{align*}
    \int_0^1 W(x,y) V_l(y) \dd y = \lambda_l V_l(x)
\end{align*}
As the operator $K$ is translation invariant in $u$, we expand perturbations into Fourier modes and restrict out attention to perturbations of the form:
\begin{align*}
    q_{m,l} = V_l(x) \cos(mu) \qquad \text{ and } \qquad q_{m,l} = V_l(x) \sin(mu)
\end{align*}
For these functions, the eigenvalue equation for an eigenvalue $\xi_{m,l}$ becomes:
\begin{align*}
    \left[ 2 \pi \beta^{-1} m^2 - \theta \pi m^2 a_m \lambda_l \right] q_{m,l} = \xi_{m,l} q_{m,l}.
\end{align*}
Therefore:
\begin{align*}
    \xi_{m,l} = \pi m^2 \left( 2\beta^{-1} - \theta a_m \lambda_l \right), \quad m \geq 1, l \in \mathbb{N} 
\end{align*}
The uniform state loses stability when the smallest of these eigenvalues crosses zero, that is at:
\begin{align*}
    \theta_c = \min_{m,l}  \frac{2}{\beta \lambda_l a_m}.
\end{align*}

\subsection{Self-consistency formulation and secondary bifurcation thresholds}\label{sec:SelfConsistency}
This subsection complements the spectral and variational analyses above by deriving stationary states from a self-consistency perspective. Starting from the stationary Fokker–Planck equation, we obtain a Gibbs-type representation of equilibria and a closed system of self-consistency equations for their Fourier modes. This framework recovers the primary bifurcation thresholds obtained earlier and further allows us to study secondary transitions that arise on non-uniform branches.

After a brief introduction, we will focus for simplicity on the bichromatic interaction potential $D(u) = -\cos(u) -2\cos(2u)$, for which the first instability is in mode $m=2$ and creates a two-peak steady state. Using the self-consistency equations, we characterise this state through a fixed-point equation, and prove that it is stable immediately above the onset, and show that it undergoes a secondary loss of stability to a one-peak state when $\theta$ is increased further.

\paragraph{Self-consistency equations}
The stationary McKean-Vlasov equation \eqref{FP} admits a Gibbs-type representation, leading to a closed system of self-consistency equations for the Fourier coefficients of the stationary density (see \cite{carrillo2020long} for a full derivation). This means that stationary solutions $\rho$ are of the form:
\begin{equation}\label{eq: Gibbs_form}
    \rho(u,x)=\frac{1}{Z(x)}\exp\left(-\beta\theta \Phi(u,x)\right),
    \qquad
    Z(x)=\int_0^{2\pi}\exp \left(-\beta\theta \Phi(u,x) \right)\dd u,
\end{equation}
where $\Phi(u,x) := \int_0^1 \int_0^{2\pi} W(x,y) D(u-v) \rho(v,y) \dd v \dd y$.
For the multichromatic interaction potential, a direct computation gives:
\begin{align*}
    \rho(u,x)=\frac{1}{Z(x)}\exp \left(\beta\theta\sum_{k=1}^n a_k R_k(x)\cos(ku)\right).
\end{align*}
where the coefficients $R_k(x)$ satisfy the self-consistency equations:
\begin{align*}
    R_k(x)=\int_0^1 \int_0^{2\pi} W(x,y) \cos(kv) \rho(v,y) \dd v \dd y,
    \qquad k=1,\dots,n,
\end{align*}
Equivalently, considering the integral operator $(Lf)(x) = \int_0^1 W(x,y)f(y) \dd y$, we can write:
\begin{align*}
    R_k = L[m_k], \qquad m_k(x) = \int_0^{2\pi} \cos(ku) \rho(u,x) \dd u,
\end{align*}
for $k \geq 1$.  
We set
\begin{align*}
    J_k := \beta \theta a_k, \qquad h_k(x) = J_k R_k(x),
\end{align*}
so that:
\begin{align*}
    \rho(u,x) \propto \exp \left( \sum_{k \geq 1} h_k(x) \cos(ku) \right).
\end{align*}
\paragraph{Linearisation about the uniform state}
Let $\rho_u = \frac{1}{2\pi}$ denote the uniform density, corresponding to $R_k \equiv 0$ for all $k$. For simplicity we work with the bichromatic potential $D(u) = -\cos(u)-2\cos(2u)$, for which the dynamics are governed by two order parameters $(R_1, R_2)$. To analyse stability and bifurcations, define the map:
\begin{align*}
    F:(R_1, R_2) \mapsto (L[m_1(h_1,h_2)], L[m_2(h_1,h_2)]). 
\end{align*}
Fixed points of $F$ correspond exactly to stationary states. Linearising $F$ around $(R_1,R_2) = (0,0)$: 
\begin{align*}
    F(R_1, R_2) = F(0,0) + DF(0,0)[(R_1,R_2)] + O(\|(R_1,R_2)\|^2).
\end{align*}
Since $F(0,0) = 0$, non-trivial solutions appear when $DF(0,0)$ has eigenvalue one.
At $h=0$, the equilibrium measure is uniform, and by orthogonality of cosine modes: 
\begin{align*}
    \frac{\partial m_j}{\partial h_l} \Bigg\rvert_{h=0} = 
    \begin{cases}
        \frac{1}{2} & \text{ if } j = l \\
        0 & \text{ if } j \neq l.
    \end{cases}
\end{align*}
As $\frac{\partial h_k}{\partial R_k} = J_k \text{Id}$, the Jacobian is:
\begin{align*}
    DF(0,0) = 
     \begin{pmatrix}
         \frac{J_1}{2} L & 0  \\
         0 & \frac{J_2}{2} L \\
    \end{pmatrix}
\end{align*}    
A bifurcation occurs when one of the blocks satisfies $\frac{J_k}{2} \lambda_1 = 1$, i.e.
\begin{align*}
    J_k^{(c)}  = \frac{2}{\lambda_1}, \qquad \theta_k^c = \frac{2}{\beta a_k \lambda_1}.
\end{align*}
The first transition corresponds to the Fourier mode with the largest coefficient $a_k$. For $D(u) = -\cos(u) - 2\cos(2u)$, this is $k=2$, producing a two-peaked stationary density which is even in $u$. Hence, near the first bifurcation, $R_1 = 0$ and $R_2 \neq 0$.
\paragraph{The even branch: $R_1 \equiv 0$}
When $R_1 = 0$, the density simplifies to 
\begin{align*}
    \rho(u,x) = \frac{1}{Z(x)} \exp(J_2 R_2(x) \cos(2u)),
\end{align*}
which is $\pi$-periodic in $u$. This implies $m_1(x) = 0$, which means $R_1 = 0$ satisfies the self-consistency equation.
The remaining equation for $R_2$ reads:
\begin{align*}
    R_2 = L \left( \frac{I_1(J_2 R_2)}{I_0(J_2 R_2)} \right),
\end{align*}
where $I_n(z) = \frac{1}{\pi} \int_0^{\pi} e^{z \cos(\theta)} \cos(n\theta) \dd \theta$ are modified Bessel functions of the first kind for integer $n \geq 0$. Using the power series expansion for $I_n$, for small $b = J_2 R_2$, we have that:
\begin{align*}
    \frac{I_1(b)}{I_0(b)} = \frac{b}{2} - \frac{b^3}{16} + \ldots,
\end{align*}
giving the expansion:
\begin{align*}
    R_2 \approx \frac{J_2}{2} L[R_2] - \frac{J_2^3}{16} L[R_2^3] + \ldots
\end{align*}
Therefore, to leading order, $R_2$ behaves like an eigenfunction of $L$ with eigenvalue $\frac{2}{J_2}$; instability occurs at $J_2^{(c)} = \frac{2}{\lambda_1}$. Near onset we write $R_2(x) \approx C_2 \phi_1(x)$, where $\phi_1$ is the principal eigenfunction of $L$. Projecting the equation onto $\phi_1$ and using self-adjointness of $L$ gives:
\begin{align*}
    \left( \frac{J_2 \lambda_1}{2} -1 \right) C_2 - K C_2^3 = 0, \qquad K := \frac{J_2^3 \lambda_1}{16} \langle \phi_1, \phi_1^3 \rangle > 0
\end{align*}
Hence:
\begin{itemize}
    \item If $\frac{J_2 \lambda_1}{2} < 1$, then the only solution is $C_2 = 0$.
    \item If $\frac{J_2 \lambda_1}{2} > 1$, then we have two additional solutions:
    \begin{align*}
        C_2 =  \pm \sqrt{\frac{J_2 \lambda_1 /2 -1}{K}}
    \end{align*}
\end{itemize}
Therefore, at the critical value of $J_2$, the uniform state $(R_1, R_2) = (0,0)$ undergoes a pitchfork bifurcation giving rise to two symmetric branches
\begin{align*}
    (0, \pm R_2(x)), \qquad R_2(x) \approx \pm C_2 \phi_1(x).
\end{align*}
\paragraph{Secondary instability: $R_1 \neq 0$}
We now investigate when the first harmonic mode $R_1$ becomes non-zero along the already bifurcated branch $R_2 \neq 0$. That is, we linearise the $R_1$-equation around $R_1=0$, while keeping $R_2(x)$ fixed from the previous step.
Expanding $m_1$ to first order in $h_1$,
\begin{align*}
    m_1(x) = \frac{\partial m_1}{\partial h_1} \Bigg\rvert_{h_1 = 0, h_2(x)} h_1(x) + O(h_1^3).
\end{align*}
The derivative $\frac{\partial m_1}{\partial h_1}$ can be computed using Bessel function identities, from which we obtain:
\begin{align*}
    g(x) := \frac{\partial m_1}{\partial h_1} \bigg \rvert_{h_1 = 0, h_2(x)} = \frac{1}{2} \left(1 + \frac{I_1(h_2(x))}{I_0(h_2(x))} \right).
\end{align*}
Substituting this linear approximation into the self-consistency equation for $R_1$,
\begin{align*}
    R_1(x) = \int_0^1 W(x,y) m_1(y) \dd y,
\end{align*}
gives
\begin{align*}
    R_1(x) \approx J_1 \int_0^1 W(x,y) g(y) R_1(y) \dd y = J_1 L[gR_1].
\end{align*}
Define the operator $A = L \circ M_g$, where $(M_g f)(y) = g(y)f(y)$ is multiplication by the function $g$. Then, the linearised equation reads:
\begin{align*}
    AR_1 = \frac{1}{J_1} R_1.
\end{align*}
A non-trivial solution $R_1 \neq 0$ exists when $\frac{1}{J_1}$ is an eigenvalue of $A$. i.e.:
\begin{align*}
    J_1 \lambda_{\max(A)} = 1 \iff \beta \theta \lambda_{\max} (A) = 1.
\end{align*}
\begin{rem}
     The operator $A$ is not self-adjoint unless $g$ is a constant, but is similar to the symmetric operator $\tilde{A} := M_{\sqrt{g}} L M_{\sqrt{g}},$ which has kernel $\tilde{W}(x,y) = \sqrt{g(x)} W(x,y) \sqrt{g(y)}$. Hence, $A$ and $\tilde{A}$ share the same spectrum and have a well-defined largest real eigenvalue.
\end{rem}
This linearisation determines the onset of the secondary instability of the even stationary states. In particular, we can now assess the stability of the two symmetric solutions $(0, \pm R_2)$ that bifurcate from the uniform state. The sign of $R_2$ only enters through the coefficient $g(x)$. The ratio $\Gamma(z) = \frac{I_1(z)}{I_0(z)}$ is odd, continuous and strictly increasing with $\Gamma(z) \in (-1,1)$, and the range of $g$ depends on the branch.
\begin{itemize}
    \item On the positive branch $(0, R_2)$, we have that $0 < \Gamma(h_2) < 1$, so $\frac{1}{2} < g(x) < 1$, and therefore:
    \begin{align*}
        \frac{1}{\beta \lambda_1} < \theta_c^{(2)} < \frac{2}{\beta\lambda_1}.
    \end{align*}
    \item On the negative branch $(0,-R_2)$, we have that $-1 < \Gamma(h_2) < 0$, so $0 < g(x) < \frac{1}{2}$, and therefore:
    \begin{align*}
        \theta_c^{(2)} > \frac{2}{\beta \lambda_1}.
    \end{align*}
\end{itemize}
The inequalities above follow by a Rayleigh-quotient inequality: since $L$ is self-adjoint with principal eigenvalue $\lambda_1$ and eigenvector $\phi_1$, and $g \in [\underline{g}, \overline{g}]$, we have $\underline{g} \lambda_1 \leq \lambda_{\max}(L M_g) \leq \overline{g} \lambda_1$.
The critical parameter $\theta_c^{(2)}$ is defined implicitly as the solution of: 
\begin{align*}
    \Phi(\theta) := \beta \theta  \lambda_{\max}(L M_{g(h_2(\theta))}) = 1.
\end{align*}
Since $\Gamma(z) = \frac{I_1(z)}{I_0(z)}$ and $h_2(\theta)$ depend continuously on $\theta$, the function $\Phi(\theta)$ is continuous on each branch. On the positive branch $(0, +R_2)$, $\lambda_{\max}(A(\theta))$ is non-decreasing and $\Phi(\theta)$ grows monotonically from $0$ to $+\infty$. Hence, there exists a unique secondary threshold $\theta_c^{(2)}$ such that $\Phi(\theta_c^{(2)}) = 1$.
On the negative branch $(0,-R_2), h_2(\theta) <0$ and $g(h_2(\theta)) \in \left( 0, \frac{1}{2} \right)$ decreases as $\theta$ increases. If $h_2(\theta)$ is bounded below, then $g$ admits a positive lower bound and $\Phi(\theta) \rightarrow \infty$, guaranteeing existence and uniqueness of $\theta_c^{(2)}$. If instead $h_2(\theta) \rightarrow -\infty$, then $g(h_2) \rightarrow 0$, and $\Phi(\theta)$ may remain below $1$ for all $\theta$, in which case the branch $(0,-R_2)$ remains linearly stable.

\subsection{Examples}
We now present some explicit formulas for the bifurcation points of the system for the graphons presented in Section \ref{subsec: setup}. Throughout these examples, we again choose the interaction potential $D(u) = -\cos(u)-2\cos(2u)$ for illustrative purposes, but the calculations apply to any interaction potential $D$ such that the relevant eigenvalues of the associated operator are simple. As established earlier, the thresholds for this potential are given by:
\begin{align*}
    \theta_c  = \frac{1}{\beta \lambda_1}, \qquad \theta_c^{(2)} = \frac{1}{\beta \lambda_{\max}(A)},
\end{align*}
where $A = L \circ M_g$, with multiplication operator $(M_g f)(x) = g(x) f(x)$, $g(x) = \frac{1}{2} \left( 1 + \frac{I_1(h_2(x))}{I_0(h_2(x))} \right)$, $h_2(x) = 2 \beta \theta R_2(x)$. The examples below follow directly by substituting the corresponding eigenvalues.
\paragraph{Erdős-Rényi graph} $W(x,y) \equiv p, p \in [0,1]$. Then the integral operator $L$ has a single eigenvalue $\lambda_1 = p, \phi_1(x) \equiv 1$. Using the formula \eqref{e: multichromatic critical}, the first critical value of $\theta$ is:
\begin{align*}
  \theta_c  = \frac{1}{\beta p}. 
\end{align*}
On the even branch $R_1 \equiv 0$, both $R_2$ and $g$ are constants, and $A = gL$ with $\lambda_{\max}(A) = gp$. Therefore, the secondary threshold is:
\begin{align}
\label{eq: secondary threshold ER}
    \theta_c^{(2)} = \frac{1}{\beta g(\theta_c^{(2)})p}, \qquad g(\theta) = \frac{1}{2} \left(1 + \frac{I_1(2 \beta \theta R_2)}{I_0(2 \beta \theta R_2) } \right) 
\end{align}
\paragraph{Power-Law graph} $W(x,y) = (xy)^{-\gamma}$, $0 < \gamma < \frac{1}{2}$.
The operator $L$ is rank one, and thus $\lambda_1 = \frac{1}{1-2\gamma}$, $\phi_1(x) = x^{-\gamma}$. Hence all solutions of the self-consistency equations take the form $R_k(x) = C_k x^{-\gamma}$ with scalar amplitudes $C_k \in \R$.
The primary threshold follows directly from the general formula:
\begin{align*}
  \theta_c = \frac{(1-2\gamma)}{\beta}
\end{align*}
On the even branch, $R_2(x) = C_2 x^{-\gamma}$, where $C_2$ is the unique non-negative solution of:
\begin{align*}
    C_2 = \int_0^1 y^{-\gamma} \frac{I_1(2\beta \theta C_2 y^{-\gamma})}{I_0(2\beta \theta C_2 y^{-\gamma})} \dd y.
\end{align*}
Since $I_1/I_0$ is smooth, strictly increasing, and satisfies $0 \leq \frac{I_1(z)}{I_0(z)} <1$, the right-hand side defines a smooth, strictly increasing function $F(C_2)$ with $F(0)=0$ and $0 < F(C_2) <(1-\gamma)^{-1}$. Hence a non-trivial solution $C_2(\theta) >0$ exists above $\theta_c $. Uniqueness follows from the fact that $\frac{F(C)}{C}$ is strictly decreasing, which is a consequence of recurrence relations and standard inequalities for Bessel functions.
Once the even branch exists, the general secondary threshold condition from Section \ref{sec:SelfConsistency} applies directly. Since $L$ is rank one, the top eigenvalue of $A = L \circ M_g$ is:
\begin{align*}
    \lambda_{\max}(A) = \int_0^{1} y^{-2\gamma} g(y) \dd y.
\end{align*}
Hence the secondary threshold is implicitly given by:
\begin{align*}
    \theta_c^{(2)} = \frac{1}{\beta \int_0^1 y^{-2\gamma} g(2\beta \theta C_2(\theta_c^{(2)})y^{-\gamma}) \dd y }
\end{align*}

\paragraph{Small-World graph} \eqref{e: SW}
A full analysis of the eigenvalues of the corresponding integral operator can be found in \cite{gao2019spectral}, and depends on the Fourier expansion of $W$. The relevant eigenvalue is $\lambda_1=2h$, (\cite{gkogkas2022graphop}, \cite{chiba2016mean}) which arises from a constant eigenfunction, so it has multiplicity 1.
The critical interaction strength is:
\begin{align*}
    \theta_c  = \frac{1}{2\beta h}
\end{align*}
As the leading eigenfunction of $L$ is constant, $R_2 \equiv C_2$, where $C_2$ solves:
\begin{align*}
    C_2 = 2h \frac{I_1(2\beta \theta C_2)}{I_0(2\beta \theta C_2)}.
\end{align*}
As before, this has a unique solution above the phase transition.
On this even branch, the first Fourier mode $R_1$ obeys the general linearised condition from Section \ref{sec:SelfConsistency}. As $g$ is a constant, the operator $A$ acts as scalar multiplication by $2hg$; hence the condition simplifies to:
\begin{align*}
    \theta_c^{(2)} = \frac{1}{2\beta h g \left( \theta_c^{(2)} \right)}, \qquad g(\theta) = \frac{1}{2} \left(1+\frac{I_1(2\beta \theta C_2(\theta))}{I_0(2\beta \theta C_2(\theta))} \right).
\end{align*}

\section{Numerical experiments}\label{sec:NumericalExperiments}
In this section, we study the long-time behaviour and the critical dynamics of the $N$-particle system \eqref{SDE}. In particular, we consider multichromatic potentials of the form:
\begin{equation}\label{e: multi potential}
    D(u) = -\sum_{k=1}^n a_k \cos(ku), 
\end{equation}
with $a_k > 0$. Such interaction potentials are regularly used to study synchronisation effects in multi-agent systems. In this context, the invariant uniform solution $\rho_u = \frac{1}{2\pi}$, corresponding to a disordered state, loses its stability in favour of peaked, ordered solutions for values of the interaction strength $\theta$ bigger than a critical threshold $\theta_c$. For this potential $D$, as proved in \eqref{e: multichromatic critical}, the critical interaction strength is given by:
\begin{align}
\label{eq: critical interaction strength}
\theta_c = \min_{m,l\in \mathbb{N}} \left\{ \frac{2}{\beta \lambda_l a_m} \right\},
\end{align}
where $\lambda_l, l \in \mathbb{N}$ are the eigenvalues of the integral graphon operator $L(V)(x)$ defined in \eqref{Loperator}.
When $n= 1$, one recovers the Kuramoto model for phase oscillators. In this case, the critical onset of synchronisation is studied by introducing the order parameter $r(t) = \frac{1}{N} \left| \sum_{j=1}^N \exp(i x_j(t)) \right| \in [0,1]$, which measures the degree of synchronisation of the $N$-particle system. In particular, $r=0$ corresponds to the disordered, uniform state $\rho_u$ and $r=1$ to full synchronisation. Intermediate values of the order parameter $r$ instead indicate the presence of a one-peaked density of oscillators, corresponding to the unique stable solution of the mean-field equation \eqref{FP}.
\\
More interesting dynamical regimes have been observed for the dynamics \eqref{SDE}, in the absence of an underlying graph structure \cite{Primi2009,Geigant2012,bertoli2025}, when more harmonics are introduced in the interaction potential. In particular, depending on the number of harmonics, the presence of long-lived multipeak densities of oscillators has been observed. Equation \eqref{eq: critical interaction strength} shows that, in these settings, the critical interaction strength is determined solely by the biggest amplitude $|a_k|$, regardless of the corresponding wavenumber $k$. Our previous analysis shows that multipeak solutions undergo secondary bifurcation characterised by the spectral properties of integral operator $L \circ M_g$, where $L$ is the graphon integral operator and $M_g$ is a suitable multiplication operator, see details in Section \ref{sec:SelfConsistency}.
\\
{  The coexistence of multiple stationary solutions of the mean-field graphon dynamics results in a dynamical evolution of the $N$-particle system with strong metastable regimes. Our numerical results  indicate that, away from the first bifurcation point, the system resides in multipeak states for long times before eventually converging towards the asymptotic state corresponding to the lowest possible wavenumber.} 
\paragraph{Order Parameters}
Identifying suitable order parameters for the investigation of critical phenomena of multi-agent systems is a fundamental issue. Order parameters are suitably designed projections of the $N$-particle system into a much lower-dimensional macroscopic subspace, which however maintains the key features of the dynamics. Most often, one is interested in reaction coordinates, special order parameters that not only provide information on the static properties of the critical dynamics, e.g. phase diagrams, but also capture dynamical features ~\cite{Rogal2021,ZagliLucariniPavliotis,Zagli_2024}. Ideally, the identification of reaction coordinates would be agnostic to the details of the dynamical evolution and obtained with data-driven techniques \cite{ZagliPavliotisLucarini2023,Evangelou2024}.   
For multichromatic interaction potentials, one usually introduces a generalisation of the Kuramoto order parameter, i.e. the set of order parameters 
\begin{equation*}
    r_k(t) = \frac{1}{N} \sum_{j=1}^N \exp\left(ikx_j\left(t\right)\right)
\end{equation*}
with $k=1,\dots,n$. It is unclear a priori which order parameter $r_k$ is best suited to investigate the critical dynamics and the dynamical metastability features of the $N$-particle systems. In this paper, we propose to use as reaction coordinate the interaction energy $\mathcal{E}_{\text{int}}$ (up to a factor of $\theta$), defined in equation \eqref{eq: interaction energy}. For the $N$-particle system with $D$ as in \eqref{e: multi potential}, the interaction energy is $\mathcal{E}_{{\text{int}}_N} = \frac{1}{N^2}\sum_{ij}^N W_{N,ij} D(X_t^i - X_t^j)$, which is simply the mean-field interaction energy $\mathcal{E}_{\text{int}}$ evaluated for the empirical measure associated to \eqref{SDE}. In the case of an all-to-all graph, associated with the constant graphon $W(x,y) =1$, the interaction energy for the multichromatic potential turns out to be \cite[Ch.5]{FrankBook}
\begin{align}
    U(t) = - \frac{1}{2} \sum_{k=1}^n |a_k| r_k(t)^2,
\end{align}
where the $r_k$ are the Kuramoto order parameters defined above. In the following, we show that $U(t)$ can be used to {  not only} pinpoint the onset of the synchronisation transition{ , but also the secondary bifurcation points}. This is not entirely surprising, because an increasing function of an order parameter remains an order parameter. However, the interaction energy reflects the contributions of all harmonics, providing a full picture of the system's behaviour. As a result, we show that $U(t)$ also captures the transitions between metastable states, which can be interpreted as a cascade towards different energy levels, see Panel \protect \subref{fig: Metastability} of Figure \ref{fig: Bi-harmonic}. We note that the remarkable features of the interaction energy as a reaction coordinate have been already highlighted for a system of interacting agents with short-range Gaussian attractive interaction potential \cite{MartzelAslangul2001}. Moreover, in the context of opinion formation models, the order parameter introduced in \cite{Chazelle_al2017a} on the basis of network-theory considerations can be interpreted as the interaction energy of the system.
\paragraph{Details on the numerical analysis} We simulate the $N$-particle system dynamics \eqref{SDE} with an Euler-Maruyama scheme with timestep $\Delta t = 0.01$. To construct the phase diagram for the energy $U(t)$ for a given graph type, say Erdős-Renyi, we perform $n_{graph}$ independent realisations of the random graph. For each random graph, we simulate $n_{noise}$ independent paths of the Wiener process in \eqref{SDE}. The initial condition for the system is always chosen to be the disordered state, i.e. $X_0^i \sim \text{Uniform}([0,2\pi])$ $\forall i = 1,\dots,N$. The energy $U(t)$ is observed for a time interval $[0,T]$. Due to the strong metastability features originating from multichromatic potentials, $T$ has to be set to a very high value when many harmonics are considered (see discussion below). The phase diagram is then constructed by averaging the asymptotic value of the energy over all simulations, namely
\begin{equation*}
    U = \left\langle \frac{1}{T - t_{tr}}\int_{t_{tr}}^T U(t)\mathrm{d}t \right\rangle,
\end{equation*}
where $\langle \cdot \rangle$ represents the average over all realisations and $t_{tr}$ is chosen to be safely within the asymptotic state. To quantify fluctuations around the mean value $U$, we also consider the following quantities
\begin{equation*}
    U_{min} = \left\langle \min_{t \in [t_{tr},T]} U(t) \right\rangle, \quad U_{max} = \left\langle \max_{t \in [t_{tr},T]} U(t) \right\rangle.
\end{equation*}
For all the systems investigated below, we set $N=1000$, $\sigma = \sqrt{2\beta^{-1}} = 0.1$ and use the interaction strength $\theta$ as the control parameter. Moreover, $n_{graph} = 5$ and $n_{noise}=3$. Regarding the graphs, we consider Erdős-Renyi (ER) graphs associated with a probability $p=0.5$, Small-World (SW) graphs constructed from a ring with $r=20$ and a rewiring probability $p=0.4$, and finally Power-Law (PL) graphs with characteristic exponents  $\gamma = 0.3$ and $\alpha = 0.4$.

\subsection{Kuramoto model}
\begin{figure}[ht]
\centering
\begin{subfigure}[b]{0.515\textwidth}
    \centering
    \includegraphics[width=\textwidth]{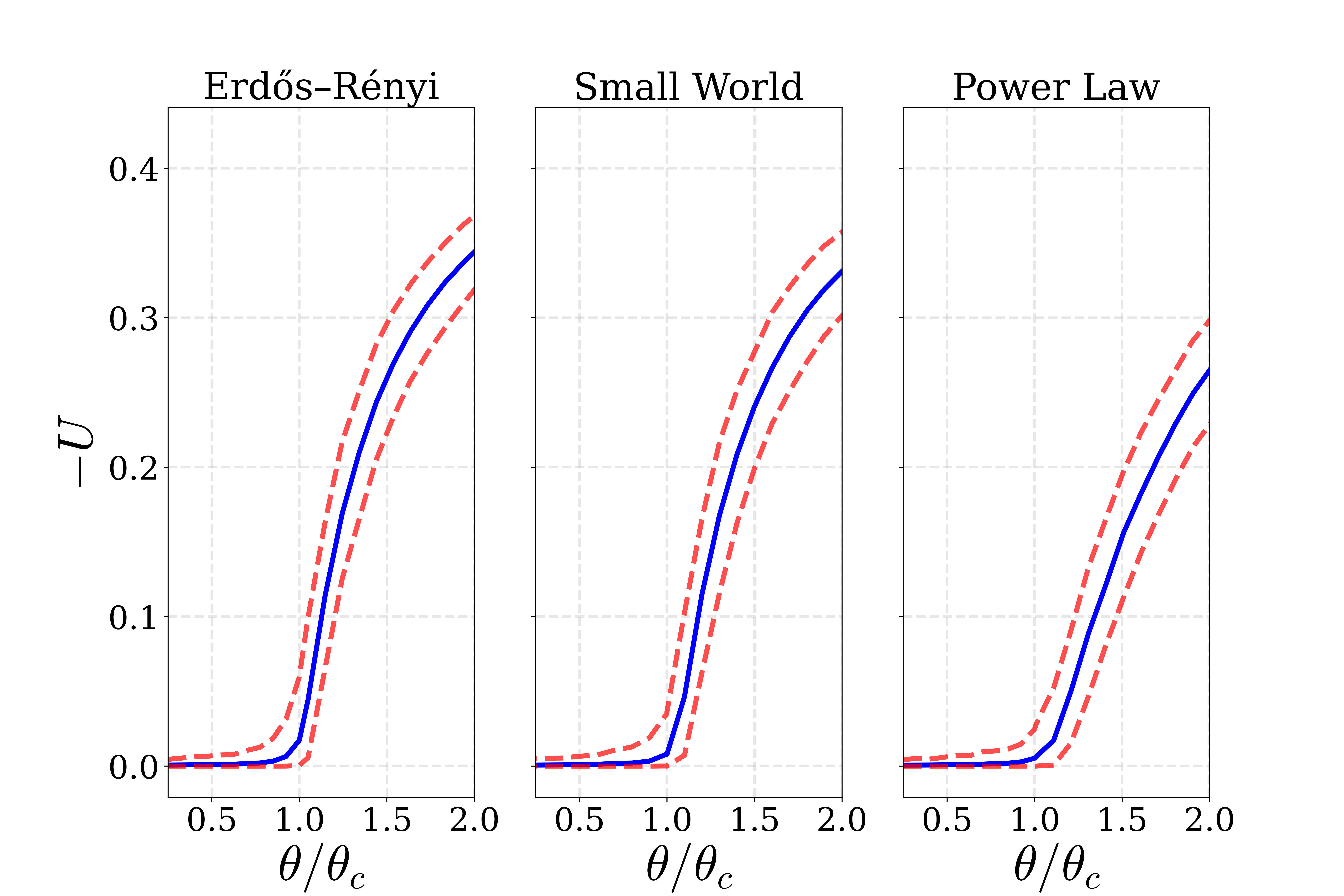}
    \caption{}
    \label{fig: Phase diag Kuramoto}
\end{subfigure}
\hfill
\begin{subfigure}[b]{0.465\textwidth}
    \centering
    \includegraphics[width=\textwidth]{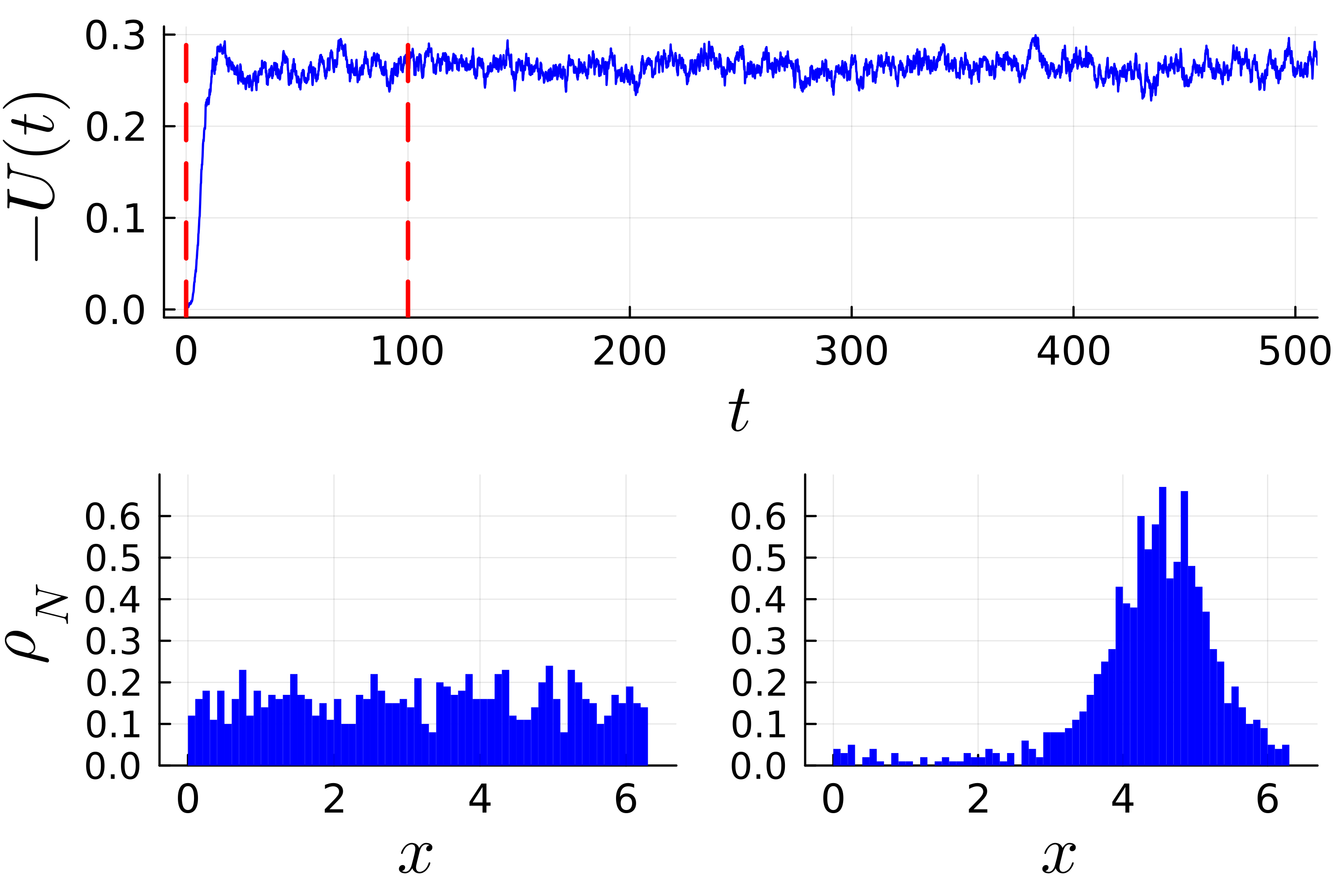}
    \caption{}
    \label{fig: No-Metastability Kuramoto}
    \end{subfigure}
\caption{(a): Phase Diagrams for the Kuramoto model. (b): (top panel) Time evolution of a typical trajectory for $U(t)$ after the phase transition, for a PL graph for $\theta / \theta_c \approx 2$. (bottom panel) Empirical measure $\rho_N$ of the system at selected times $t=0,100$, represented as red vertical dashed lines in the top panel. }
\label{fig: Kuramoto}
\end{figure}
The Kuramoto model corresponds to the single harmonic potential $D(u) = -\cos(u)$. In this case, the disordered state $\rho_u = \frac{1}{2\pi}$ is no longer stable when $\theta / \theta_c  > 1$ and an ordered, one-peak state originates from it. The phase diagrams of the interaction energy $U(t) = - |r_1(t)|^2$ of the system for different graph topologies are shown in panel \protect\subref{fig: Phase diag Kuramoto} of Figure \ref{fig: Kuramoto}. 
The phase diagram has been evaluated for $T = 1000$ and $t_{tr} = 800$ which we found to be appropriate for all values of the interaction strength considered. The Kuramoto model has been extensively studied in the literature and our results agree with \cite{gkogkas2022graphop} for the ER and SW graphs, and with \cite{chiba2018bifurcations} (in the absence of diffusion) for the PL graph. 
\\
The Kuramoto model does not exhibit any metastable features. In panel \protect\subref{fig: No-Metastability Kuramoto} of Figure \ref{fig: Kuramoto} (top panel) we show the typical evolution of the energy after the phase transition, together with the empirical measure (bottom panels) $\rho_N$ of the system at selected times (represented as red, vertical, dashed lines). The system, uniformly distributed on the torus at time $t = 0$, reaches very quickly ($t \approx 50$) an ordered, peaked state characterised by a non-vanishing energy. 
\subsection{Bichromatic potential}
\begin{figure}[ht]
\centering
\begin{subfigure}[b]{0.515\textwidth}
    \centering
    \includegraphics[width=\textwidth]{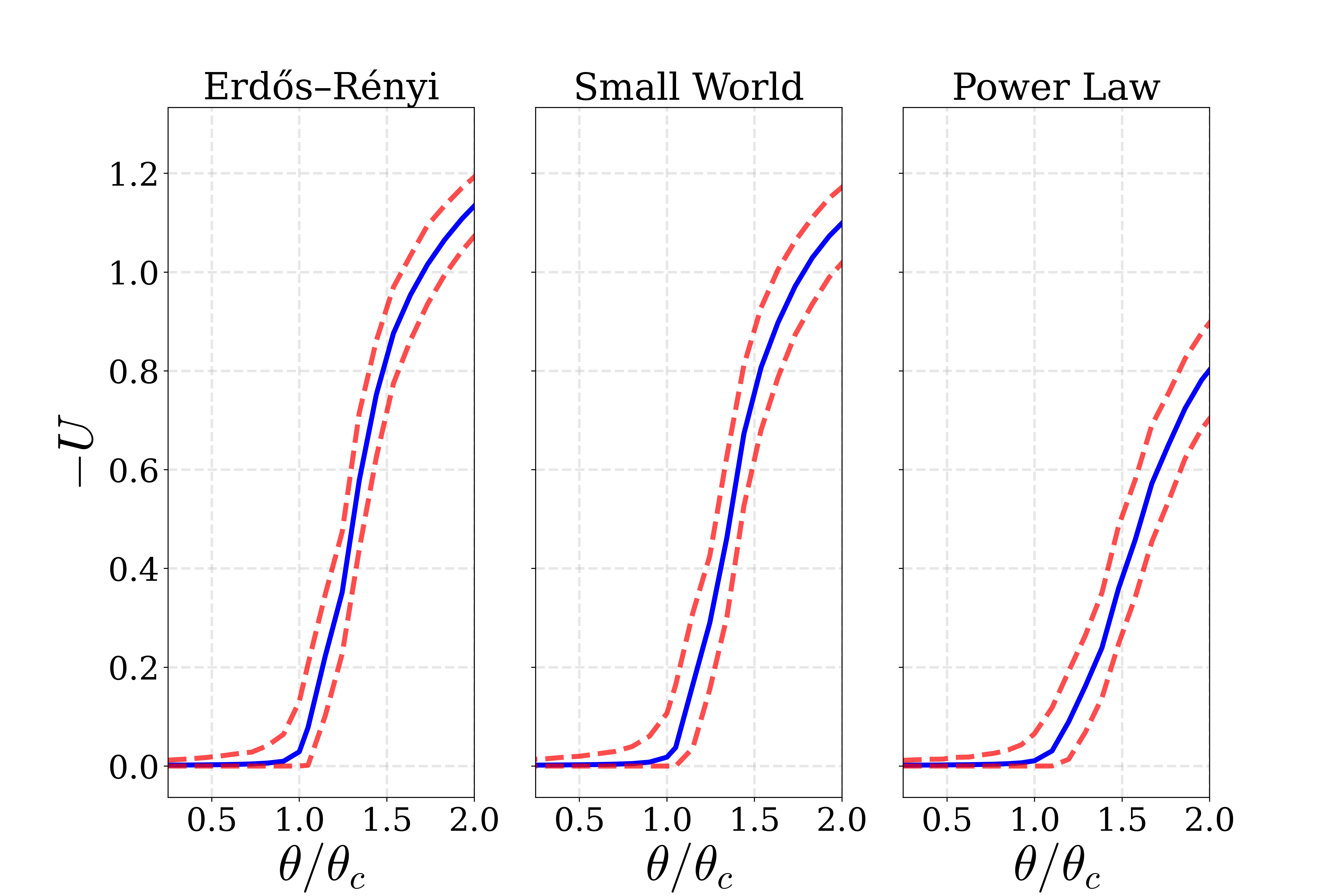}
    \caption{}
    \label{fig: Phase diag Bi-harmonic}
\end{subfigure}
\hfill
\begin{subfigure}[b]{0.475\textwidth}
    \centering
    \includegraphics[width=\textwidth]{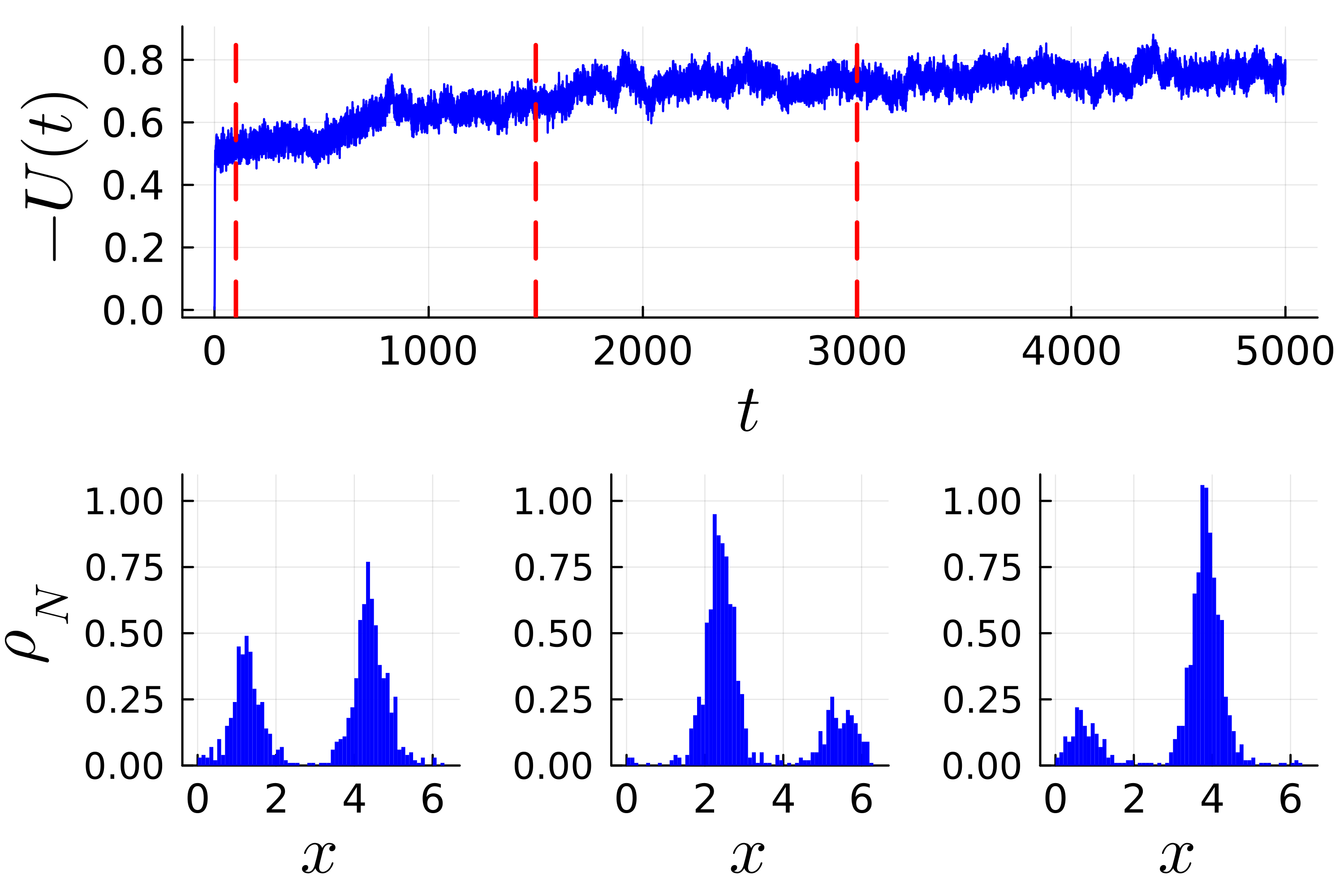}
    \caption{}
    \label{fig: Metastability}
    \end{subfigure}
\caption{(a): Phase Diagrams for the bi-harmonic interaction potential. (b): (top panel) Time evolution of a typical trajectory for $U(t)$ after the phase transition, for a PL graph for $\theta / \theta_c \approx 1.96$. (bottom panel) Empirical measure $\rho_N$ of the system at selected times $t=100,1500,3000$. Graphical conventions as in Figure \ref{fig: Kuramoto}.}
\label{fig: Bi-harmonic}
\end{figure}
Here, we consider a bichromatic potential $D(u) = - \cos(u) - 2 \cos(2u)$. The introduction of the new harmonic not only changes the critical value of the interaction strength $\theta_c$ but also considerably impacts the overall dynamics of the system. The phase diagrams for the energy $U$, corresponding to $T = 5000$ and $t_{tr} = 4500$, are provided in panel \protect \subref{fig: Phase diag Bi-harmonic} of Figure \ref{fig: Bi-harmonic}. The numerical results corroborate the theoretical prediction for the critical interaction strength \eqref{eq: critical interaction strength}. 
{  As explained in details in section \ref{sec:SelfConsistency}, the uniform solution loses stability at the critical interaction strength $\theta_c$, giving rise to a double peaked solution characterised by \textit{microscopic order parameters} $(R_1(x),R_2(x)) = (0,\pm R^*_2(x))$, solution of a suitable set of self-consistency equations. The branch of the double peaked solution characterised by a positive value of $R_2^*(x)$ undergoes a secondary bifurcation, from which a single peak solution originates. We have here investigated the second bifurcation point for the bi-chromatic potential on the Erdős-Rényi graph for which the second critical interaction strength can be numerically estimated as a solution of the self consistency equation \eqref{eq: secondary threshold ER}. Figure \ref{fig: ER secondary bifurcartion} shows the phase diagram over a much finer grid of interaction strength values past the first bifurcation point. It is clear that there is a change in dynamical regime around the secondary critical threshold $\theta_c^{(2)}$ represented as a vertical dashed line.}
Furthermore, the bichromatic potential presents strong metastability features, with the typical timescale needed to reach the stationary state being more than one order of magnitude bigger than for the Kuramoto model.
Panel \protect \subref{fig: Metastability} of Figure \ref{fig: Bi-harmonic} shows the typical evolution of the energy for settings similar to what is shown in Figure \ref{fig: Kuramoto}. Firstly, the system quickly reaches a two-peak state approximately at $t=100$, characterised by an energy $-U \approx 0.5$. Such a state is long-lived but appears to be metastable: we observe a transition to a lower energy level at $t \approx 1000$. Right after the transition, the profile of the empirical measure indicates that most of the oscillators have transitioned towards one of the two peaks. Following the transition, the system exhibits a slower dynamics where particles keep leaking from the small peak to the other peak, which becomes narrower. A similar evolution, with peaks exchanging mass, has been observed in an aggregation model featuring metastable states \cite{Evers2016}. 
\begin{figure}[ht]
    \centering
    \includegraphics[width=0.5\linewidth]{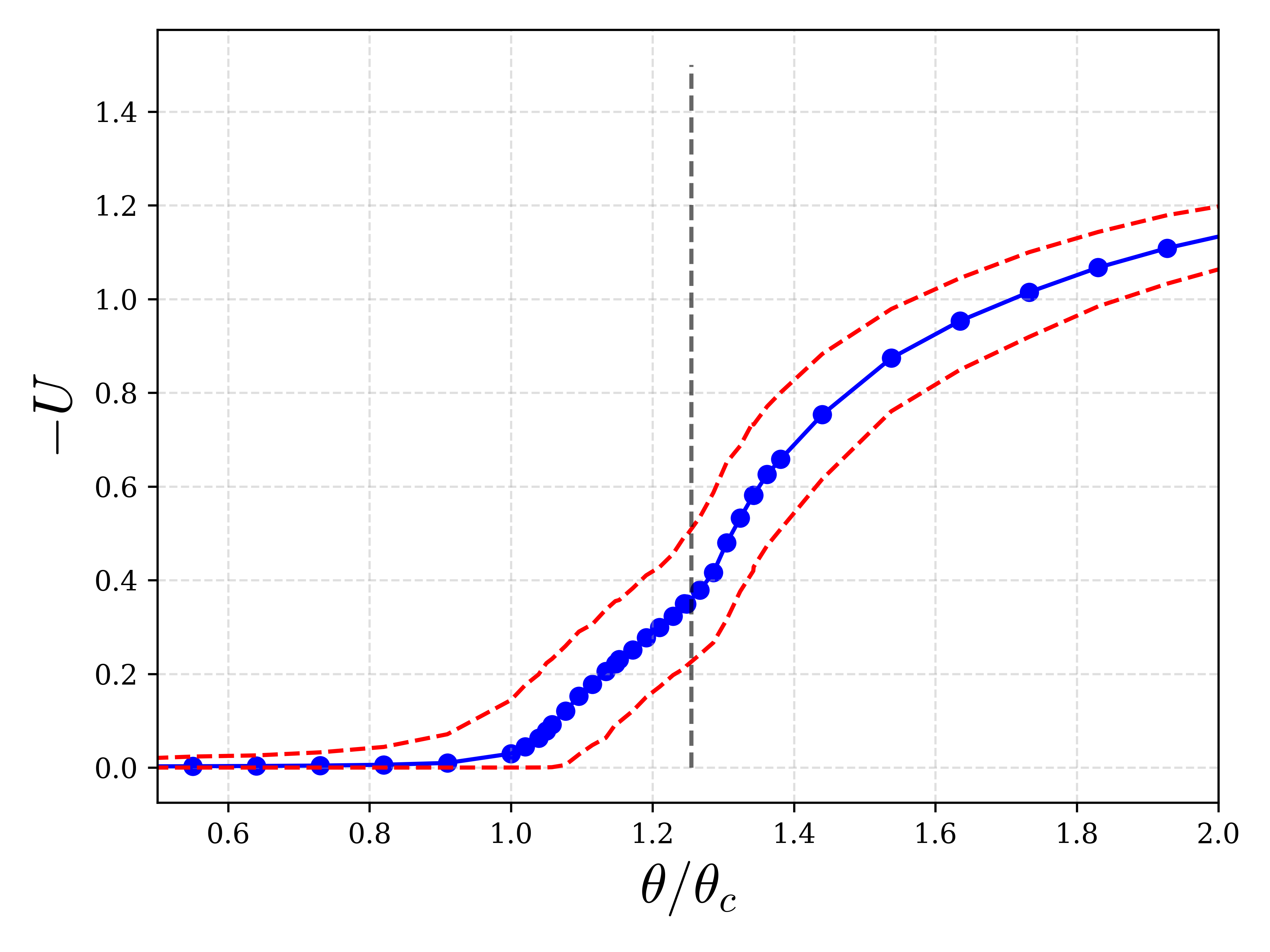}
    \caption{Phase diagram for the bi-harmonic interaction potential on Erdős-Rényi graph with a finer interaction strength grid. The vertical dashed line represent the analytical value of the secondary bifurcation point. Other graphical conventions as in Figure \ref{fig: Kuramoto}.}
    \label{fig: ER secondary bifurcartion}
\end{figure}
\subsection{Quadrichromatic potential}
\begin{figure}[ht]
\centering
\begin{subfigure}[b]{0.515\textwidth}
    \centering
    \includegraphics[width=\textwidth]{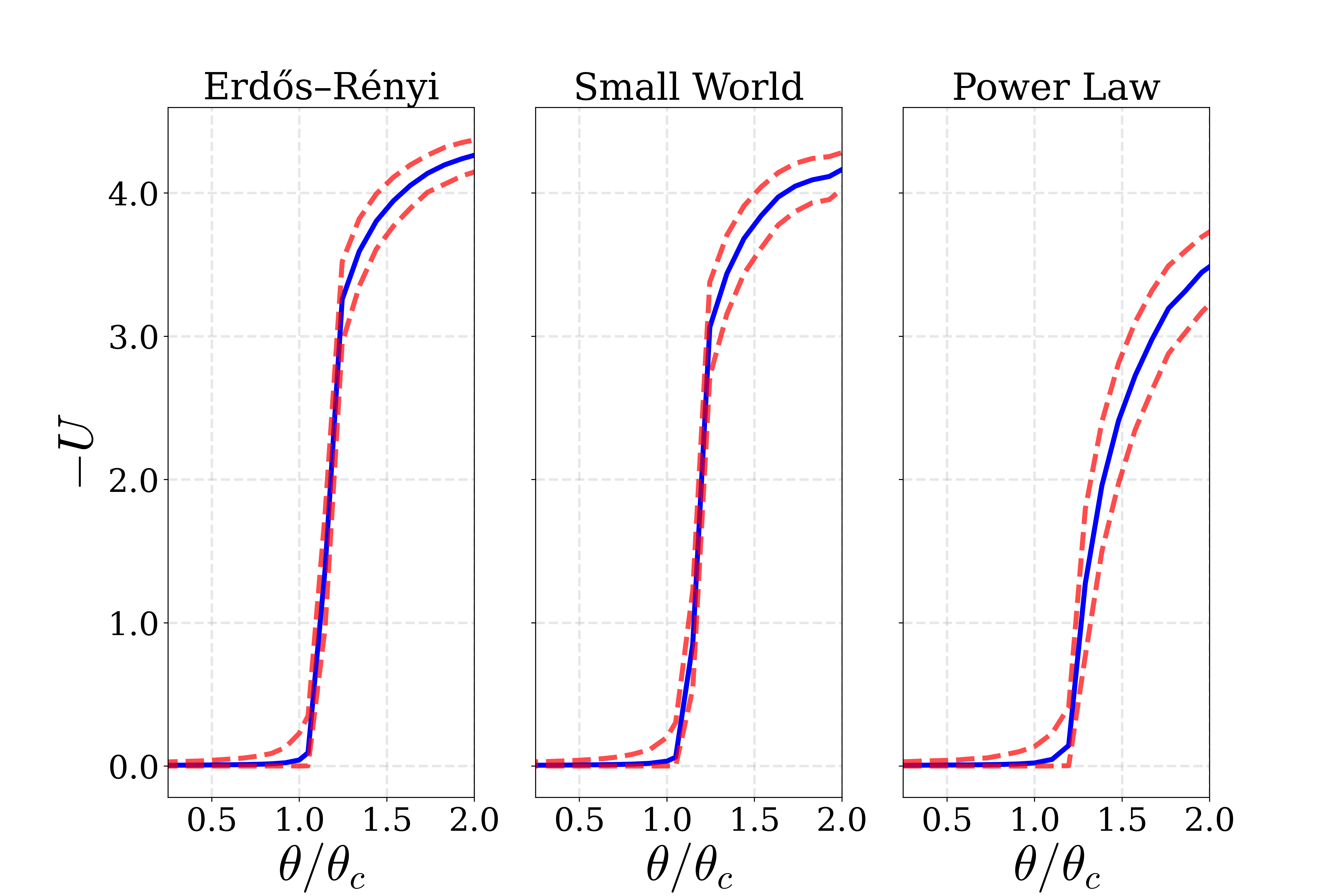}
    \caption{}
    \label{fig: Phase diag Quadri}
\end{subfigure}
\hfill
\begin{subfigure}[b]{0.475\textwidth}
    \centering
    \includegraphics[width=\textwidth]{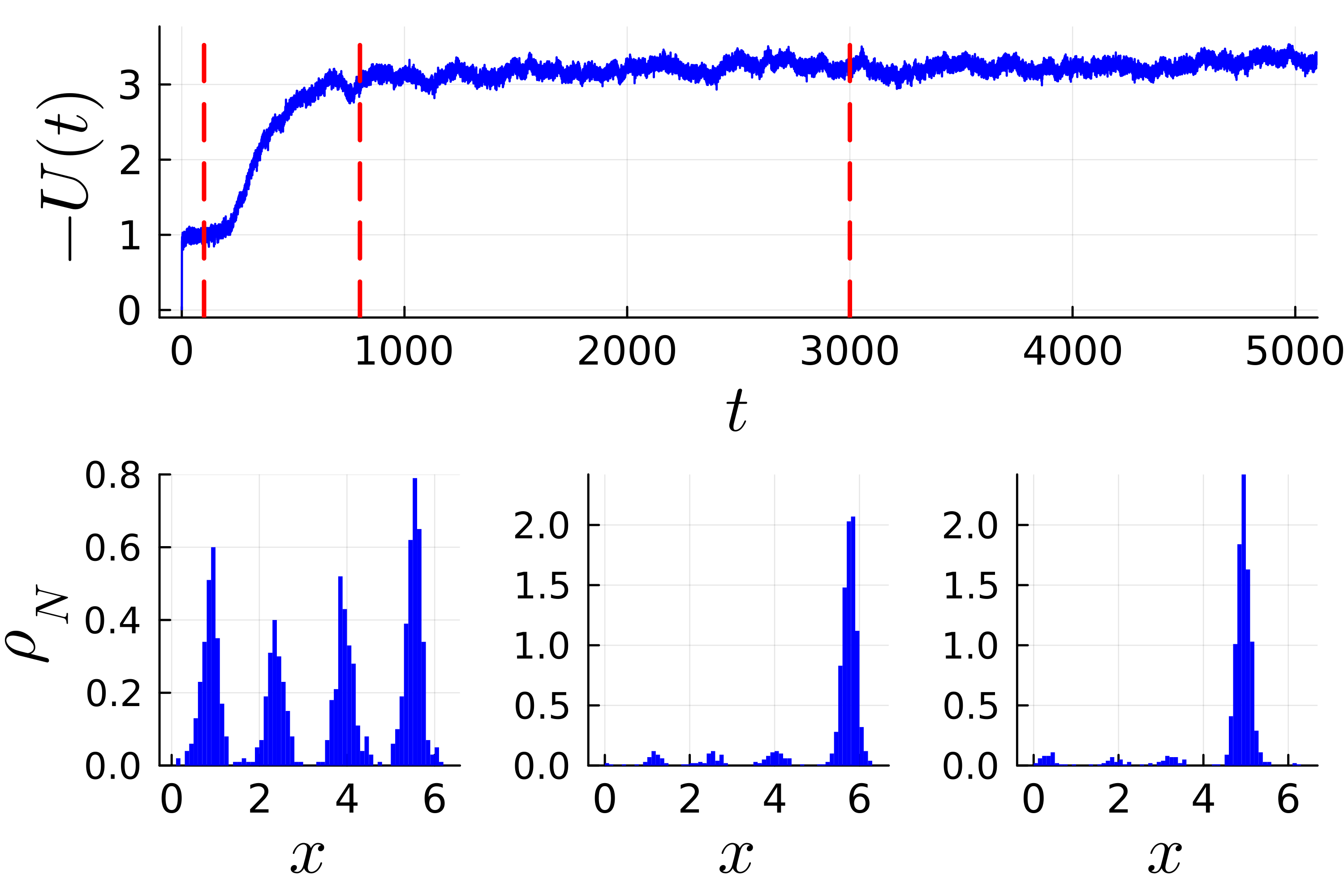}
    \caption{}
    \label{fig: Metastability Quadri}
    \end{subfigure}
\caption{(a): Phase Diagrams for the quadri-harmonic interaction potential. (b): (top panel) Time evolution of a typical trajectory for $U(t)$ after the phase transition, for a PL graph for $\theta / \theta_c \approx 1.96$. (bottom panel) Empirical measure $\rho_N$ of the system at selected times $t=100,800,3000$. Graphical conventions as in Figure \ref{fig: Kuramoto}.}
\label{fig: Quadri}
\end{figure}
Here, we consider the dynamics prescribed by the quadrichromatic potential $D(u) = - \cos(u) - 2 \cos(2u) - 3\cos(3u) -4\cos(4u)$. As with the bichromatic interaction, the potential $D(u)$ has several local minima and a global minimum at $u=0$. Consequently, we expect that the $4-$peak solution is metastable, and that it persists over long time intervals. In fact, the more negative Fourier modes we add to the interaction potential, the closer we get to the case where the system exhibits a discontinuous phase transition, since the resonance condition from~\cite[Thm 1.3(b)]{carrillo2020long} is almost satisfied. Therefore, it is not surprising that the dynamics is dominated by dynamical metastability, a common feature of systems exhibiting discontinuous phase transitions.

Panel \protect \subref{fig: Phase diag Quadri} of Figure \ref{fig: Quadri} shows the phase diagram of the energy $U$, and corroborates our theoretical results regarding the value of the critical interaction strength $\theta_c$ given by \eqref{eq: critical interaction strength}. 
As opposed to the previous sections, here we observe a less smooth change in the steepness of the curve, with an initial slow increase of the energy $U$ near $\theta / \theta_c =  1$ followed by a sudden steep increase. 
This is due to the strong metastability properties exhibited by the quadrichromatic potential, which complicates the numerical investigation of the stationary properties of the system near the phase transition.
On the one side, panel \protect \subref{fig: Metastability Quadri} shows that, far from the phase transition ($\theta / \theta_c \approx 1.96$), a typical energy trajectory will initially fluctuate around $-U \approx 1$ and then transition to a lower energy state $-U \approx 3$. 
The empirical measure of the system is characterised by four peaks in the metastable state, whereas its asymptotic profile is characterised by a single, clustered state. This provides further numerical evidence that the stable solution of the $N-$particle system with a multichromatic potential is a one-peak density of particles.
\begin{figure}[ht]
    \centering
    \begin{subfigure}[b]{0.49\textwidth}
        \centering
        \includegraphics[width=\textwidth]{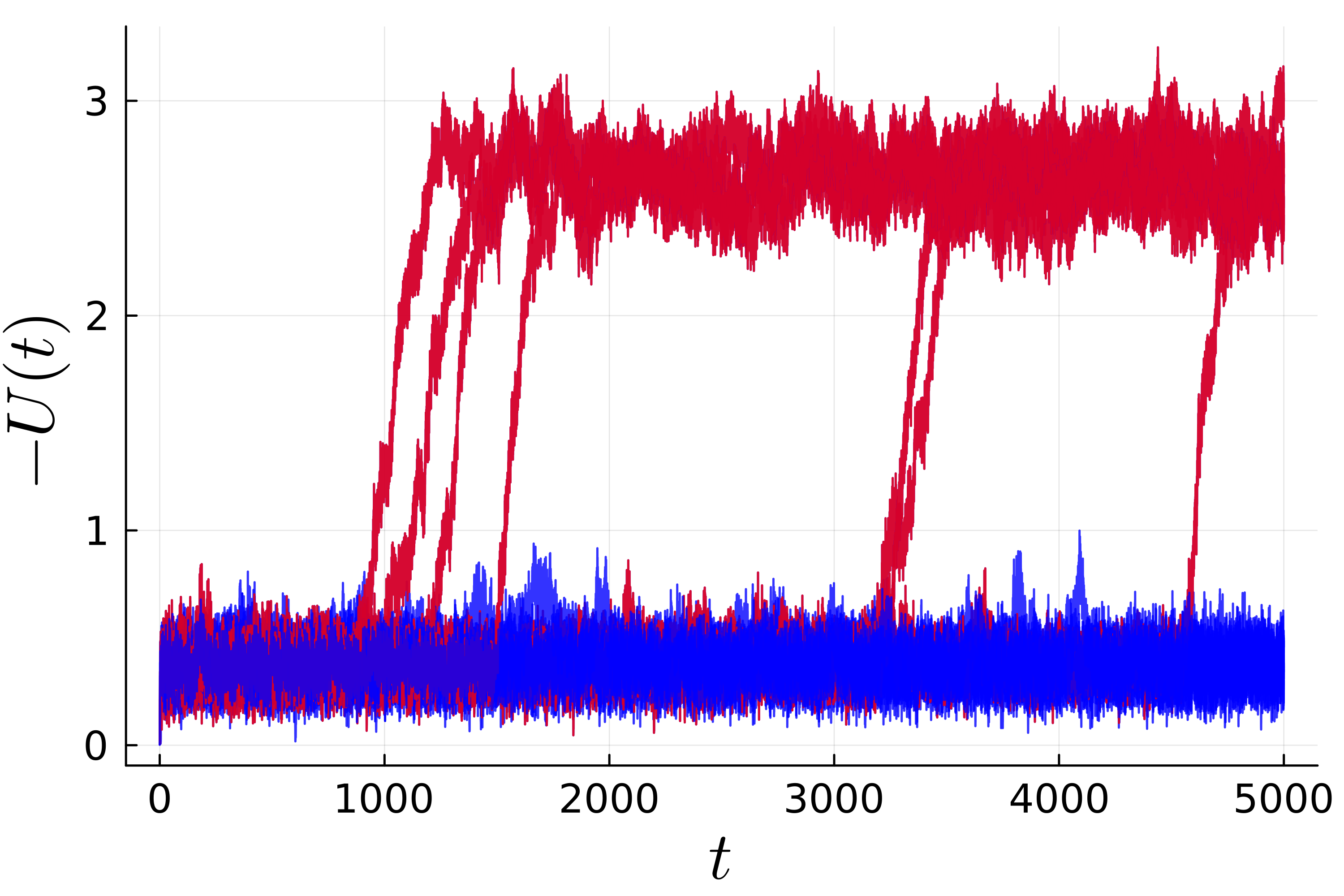}
        \caption{}
        \label{fig: Energy ER near Phase}
    \end{subfigure}
    \hfill
    \begin{subfigure}[b]{0.49\textwidth}
        \centering
        \includegraphics[width=\textwidth]{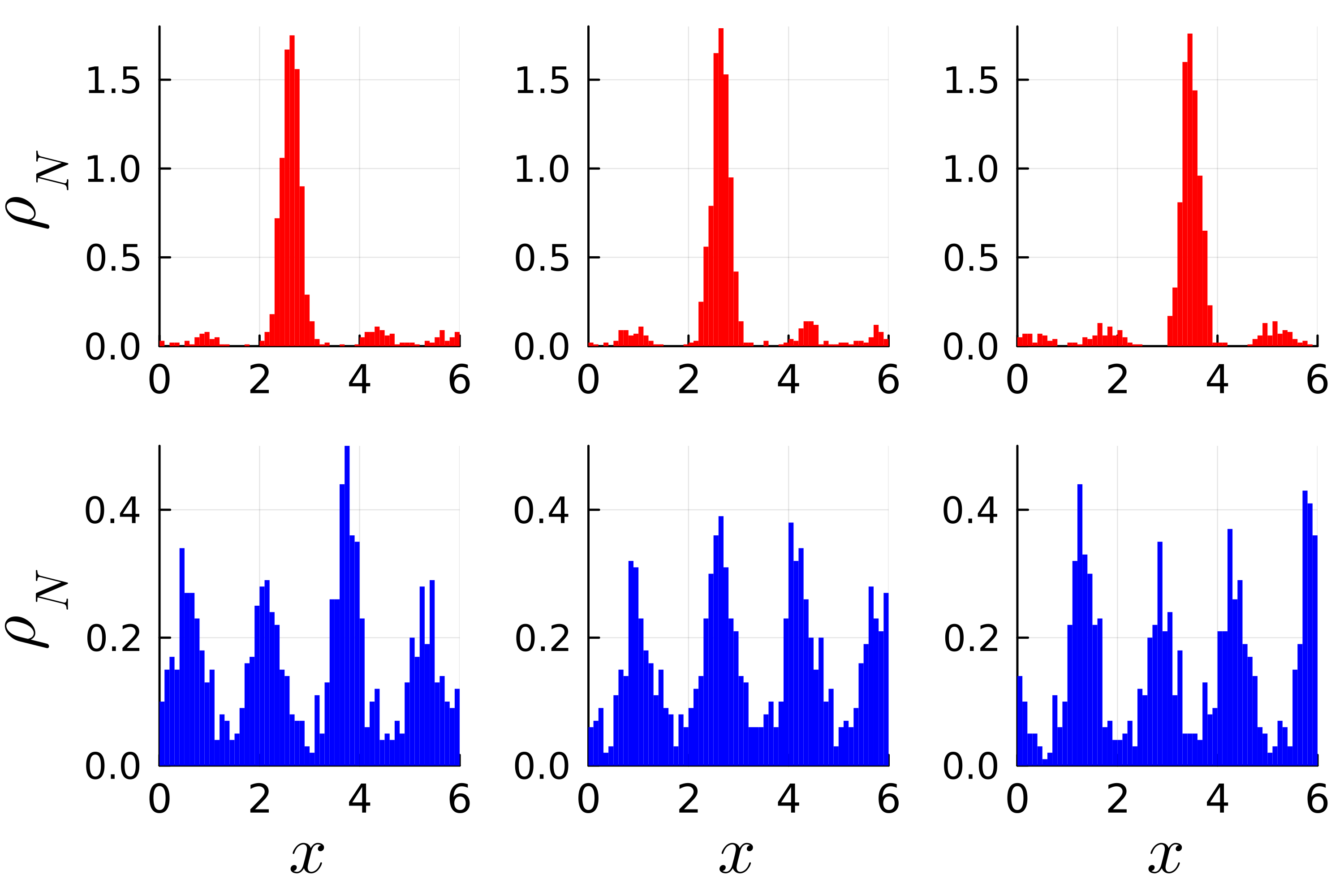}
        \caption{}
        \label{fig: Emp Measure ER}
        \end{subfigure}
    \caption{Metastability features for the quadri-harmonic potential on ER graphs. Panel (a): Energy of the system as a function of time. Panel (b): Empirical measure of the system at the end of the simulation time $T=5000$ for selected trajectories. Trajectories that have (not yet) transitioned to the final one-peak state are represented in red (blue). Here, $\theta / \theta_c \approx 1.14$}
    \label{fig: Metastability_Near_Phase_ER}
    \end{figure}
On the other side, just above the phase transition, the dynamics is slower due to the critical slowing-down characterising continuous phase transitions. 
Here, the phase diagrams have been obtained by setting $T = 5000$ for the ER and SW graphs and $T = 10000$ for the PL graph (in all cases $t_{tr} = 0.9 T$). 
For all values of $\theta$ far from the phase transition point, we have found this to be a good choice, as it allows ample time for the system to reach its asymptotic, one-peak state.  
Panel \protect \subref{fig: Energy ER near Phase} of Figure \ref{fig: Metastability_Near_Phase_ER} shows typical trajectories of the energy $U(t)$ for the ER graph and settings near the phase transition ($\theta / \theta_c \approx 1.14$). 
We observe that, fixed $T = 5000$, some trajectories (in red), initially fluctuating around $-U \approx 0.4$, have transitioned to a lower energy state $-U \approx 2.7$. In contrast, other trajectories (in blue) have yet to make the transition. 
The lower energy state corresponds to the one-peak state as shown in red in panel \protect \subref{fig: Emp Measure ER}, whereas the trajectories that have not yet escaped the metastable state are characterised by a four-peak empirical measure (in blue), which is extremely long-lived due to the critical slowing down. 
One could potentially construct a \textit{rectified} phase diagram by averaging the energy only on the trajectories that have reached the one-peak state. Still, we have preferred here to consider all trajectories to highlight the important effects of the long-lived, metastable states. 
Interestingly, for similar values of $\theta / \theta_c$, we have observed no trajectories escaping the four-peak metastable state for the PL graph. 
A careful analysis of the statistics of escape times and metastability properties of the system would go beyond the scope of this paper and we leave it for future work.  

\section{Conclusions}\label{sec:Conclusions}
In this paper, we studied the effect of the underlying (random) graph topology on phase transitions for mean-field limits of stochastic interacting particle systems on random graphs. We first analysed the structure and properties of the mean-field PDE, including the existence of a gradient flow structure in an appropriate metric space and the properties of its associated free energy. We then showed, by extending the Crandall-Rabinowitz-style bifurcation theory from \cite{tamura1984asymptotic}, that the mean-field system has a bifurcation point at a specified critical interaction strength, which depends both on the interaction potential and on the underlying graph structure. The study of bifurcations and, in particular, the identification of the critical interaction strengths for primary and secondary bifurcations of mean-field solutions is based on spectral analysis of the graphon integral operator. In addition, we complemented this spectral analysis with a self-consistency formulation that provides an alternative perspective on the stationary states and bifurcations. This framework revealed secondary transitions on non-uniform branches, offering a more complete characterization of the stability landscape of multichromatic interaction potentials on random graphs. We applied our theoretical findings to several examples of random graphs and interaction potentials.
Finally, we performed extensive, highly resolved, numerical simulations of the $N$-particle system; in particular, we explored the dynamical metastability of interacting particle systems with multichromatic interaction potentials on random graphs.

The work presented in this paper can be extended in several directions. We mention here a few problems that we are currently exploring. First, it would be interesting to consider the effect of a confining potential, together with the graphon structure. Second, we can consider the underdamped Langevin dynamics and study the effect of inertia (in particular, in the low friction regime) on the dynamics. The detailed analysis of the stability of different stationary states via the study of the spectrum of the linearized McKean-Vlasov operator, extending the results from~\cite{bertoli2025} is also of interest. As demonstrated in this work, interacting particle systems on graphs that exhibit phase transitions are characterized by dynamical metastability. The rigorous, systematic study of this phenomenon is a topic of great interest. 

\paragraph{Acknowledgments} BB is funded by a studentship from the Imperial College London EPSRC DTP in Mathematical Sciences Grant No. EP/W523872/1. GP is partially supported by an ERC-EPSRC Frontier Research Guarantee through Grant No. EP/X038645, ERC Advanced Grant No. 247031 and a Leverhulme Trust Senior Research Fellowship, SRF$\backslash$R1$\backslash$241055. NZ has been supported by the Wallenberg Initiative on Networks and Quantum Information (WINQ). GP is grateful to Andr\'{e} Schlichting and Rishabh Gvalani for useful discussions.
\bibliographystyle{plain} 
\bibliography{biblio}
\end{document}